\documentclass[a4paper]{IEEEtran}
\usepackage{amsmath,amsthm,amsfonts,amssymb}
\usepackage{tikz}
\usepackage{enumitem}
\usepackage{newfloat}
\usepackage[hidelinks,bookmarks=false]{hyperref}
\hypersetup{
   pdfauthor={E.P.Csirmaz and L.Csirmaz},
   pdftitle={Attempting the impossible: enumerating extremal submodular functions for n=6},
   bookmarksopen=false,
   pdfkeywords={vertex enumeration, supermodular function},
}
\makeatletter
\def\correcthref{\hyper@anchor{\@currentHref}}
\makeatother

\newcommand\R{\mathbb R}
\newcommand\MOD{\mathsf{MOD}}
\newcommand\setm{\smallsetminus}
\def\c{\complement}
\let\phi\varphi

\newcommand\matrixrestr[2]{#1[#2]}
\newcommand{\mof}[1]{\matrixrestr{M}{#1}}
\newcommand\mofr{\matrixrestr{M}{r}}
\newcommand\mofrp{\matrixrestr{M}{r'}}
\newcommand\mpofrp{\matrixrestr{M'}{r'}}
\newcommand\zerovec{\mathbf{0}}

\newcommand\tailo{t-opt}
\newcommand\tailorder{\tailo{} order}
\newcommand\tailinsorder{\tailo{} insertion order}
\newcommand\Tailinsorder{T-opt insertion order}

\renewcommand\;{\tmspace+{2mu}{.125em}}

\let\hat\widehat
\newcommand\halfg{\raisebox{-1.2pt}{\textonehalf}\kern0.3pt}
\newcommand\rstd{\mathrel{\prec}}
\DeclareMathOperator\REC{\mathsf{SUB}}

\newtheorem{theorem}{Theorem}
\newtheorem{claim}[theorem]{Claim}

\theoremstyle{definition}
\newtheorem{definition}{Definition}
\usepackage[noEnd=false,indLines=true,beginComment=~,commentColor=black!90,
   beginLComment=--~,endLComment=]{algpseudocodex}
\tikzset{algpxIndentLine/.style={draw,line width=0.5pt,color=gray!60}}
\makeatletter
\algnewcommand\algorithmicforeach{\textbf{for each}}
\algdef{S}[FOR]{ForEach}[1]{%
        \algpx@startIndent\algpx@startCodeCommand\algorithmicforeach\ #1\ \algorithmicdo%
}
\pretocmd{\ForEach}{\algpx@endCodeCommand}{}{}
\algdef{SE}[LOOP]{Loop}{EndLoop}{%
       \algpx@startIndent\algpx@startCodeCommand\algorithmicloop%
}{%
        \algpx@endIndent\algpx@startEndBlockCommand\algorithmicend%
}
\algdef{SE}[IF]{If}{EndIf}[1]{%
        \algpx@startIndent\algpx@startCodeCommand\algorithmicif\ #1\ \algorithmicthen%
}{%
        \algpx@endIndent\algpx@startEndBlockCommand\algorithmicend%
}
\algdef{SE}[FOR]{For}{EndFor}[1]{%
        \algpx@startIndent\algpx@startCodeCommand\algorithmicfor\ #1\ \algorithmicdo%
}{%
        \algpx@endIndent\algpx@startEndBlockCommand\algorithmicend%
}
\algdef{SE}[WHILE]{While}{EndWhile}[1]{%
    \algpx@startIndent\algpx@startCodeCommand%
    \algorithmicwhile\ #1\ \algorithmicdo%
}{%
     \algpx@endIndent\algpx@startEndBlockCommand\algorithmicend}%

\makeatother
\algrenewcommand\alglinenumber[1]{\footnotesize #1~}
\DeclareFloatingEnvironment[]{code}
\newenvironment{pseudocode}[1]
{\begin{code}[!thb]%
  \hrule\vskip-2pt
  \caption{\fontsize{9}{10}\selectfont #1}%
  \vskip 2pt
  \hrule\hbox{}\begin{algorithmic}[1]\ignorespaces}
{\end{algorithmic}
  \vskip 3pt \hrule
  \end{code}}

\newenvironment{itemz}[1][3pt]{%
\setitemize{topsep=#1,noitemsep,leftmargin=1.5\parindent,labelwidth=1.2\parindent,labelsep=3pt,align=parleft}%
\begin{itemize}}{\end{itemize}}

\newenvironment{Keywords}[1]{\xdef\IEEEkeywordsname{#1}\IEEEkeywords}{\endIEEEkeywords}

\title{\bfseries\fontsize{17.4}{19}\selectfont
 Attempting the impossible: enumerating extremal submodular functions for n=6
}

\newcommand{\rot}[3]{#3#2#1} %
\author{\fontsize{12.4}{12}\selectfont
Elod P. Csirmaz\IEEEauthorrefmark1
\thanks{\IEEEauthorrefmark1e-mail: \rot{\rot{maz.}{csir}{ep}com}{@}{elod},
\space R\'enyi Institute, Budapest}
and
Laszlo Csirmaz\IEEEauthorrefmark2
\thanks{\IEEEauthorrefmark2e-mail: csirmaz@renyi.hu,
\space UTIA, Prague and R\'enyi Institute, Budapest}}

\begin{document}
\maketitle

\begin{abstract}
Enumerating the extremal submodular functions defined on subsets of a fixed
base set has only been done for base sets up to five elements. This paper
reports the results of attempting to generate all such functions on a
six-element base set. Using improved tools from polyhedral geometry, we have
computed 360 billion of them, and provide the first reasonable estimate of
their total number, which is expected to be between 1,000 and 10,000 times
this number. The applied Double Description and Adjacency Decomposition
methods require an insertion order of the defining inequalities. We
introduce two novel orders, which speed up the computations significantly,
and provide additional insight into the highly symmetric structure of
submodular functions. We also present an improvement to the combinatorial
test used as part of the Double Description method, and use statistical
analyses to estimate the degeneracy of the polyhedral cone used to describe
these functions. The statistical results also highlight the limitations of
the applied methods.


\end{abstract}


\section{Introduction}

Submodular functions are analogues of convex functions that enjoy numerous
applications. Their structural properties have been investigated
extensively, and they have applications in such diverse areas as information
inequalities, operational research, combinatorial optimization and social
sciences, and they have also found fundamental applications in game theory
and machine learning. Consult \cite{balcan2018} and the references therein
for additional examples. For a comprehensive overview of how submodular
functions are used in machine learning in particular and in optimization in
general, see Bach's excellent monograph \cite{bach2013}.

This paper reports our results in attempting to generate all extremal
submodular functions defined on the subsets of a base set of size $n=6$.
Extremal functions form a unique minimal generating set of all submodular
functions, thus their knowledge provides invaluable information about their
structure. For $n=5$ the complete collection of extremal submodular
functions was first reported in \cite{FiveVars}. Our aim was to develop and
implement techniques which can handle the significantly more difficult
problem of listing these functions for $n=6$. Using the developed methods we
succeeded in computing the first 360 billion members of the $n=6$ list, and
also in giving a reasonable estimate for their total number.

The paper is organized as follows. Sections \ref{sec:funcintro} and
\ref{sec:1} provide definitions and an introduction to submodular functions
over finite base sets, and their description as a high-dimensional
polyhedral cone. Section \ref{sec:eer} provides an overview of methods from
polyhedral geometry used in our computations, including the Double
Description method, and the algebraic and combinatorial tests for the
adjacency of rays of the cone. Section \ref{sec:bitmap} introduces
additional improvements to the combinatorial test.

Section \ref{sec:apply} details our approaches to applying these tools to
the $n=6$ problem, including a description of the degree of degeneracy of
the polyhedral cone, and introducing two novel orderings of the inequalities
defining the cone. These orderings allowed the Double Description method to
proceed farther with considerably lower resource requirements. These new
orders, which we named \emph{t-opt} and \emph{recursive}, provided
additional insights into the intricate and highly symmetric structure of
these polyhedral cones.

Section \ref{sec:apply} continues with our results and statistical analyses
based on two approaches to generate a high number of extremal submodular
functions, while Section \ref{sec:estimate} contains an estimate for their
overall number and for the total number of their (symmetrical) orbits.

Implementations of the main algorithms discussed in this paper can be found at
\href{https://github.com/csirmaz/submodular-functions-6}%
{https://github.com/csirmaz/submodular-functions-6}. Partial results of the
computations representing 360 billion extremal submodular functions in 260M
orbits are available as a Zenodo Dataset at 
\url{https://zenodo.org/records/13954788}.

\section{Submodular functions}\label{sec:funcintro}

In general, submodular functions are real-valued functions defined on some
lattice. In the most important case---which is also the subject of this
paper---the lattice is formed by all subsets of a (finite) set $X$, called
the \emph{base}, with the intersection and union as lattice operations.
While definitions are spelled out for this special setting, those in general
are similar. Some other lattice-based functions are discussed as
illustrations.

The functions we are interested in assign real numbers
to subsets of a (typically finite) base set $X$. Such a function $f$ is called
\emph{modular} if
$$
    f(A)+f(B)=f(A\cap B)+f(A\cup B)
$$
holds for all subsets $A$ and $B$ of $X$. Functions satisfying
\begin{equation}\label{eq:2}
   f(A)+f(B) \ge f(A\cap B)+f(A\cup B)
\end{equation}
for all $A$ and $B$
are called \emph{submodular}, referring to the fact that the right hand side
is below what modularity would require. Similarly, functions satisfying
$$
  f(A)+f(B) \le f(A\cap B)+f(A\cup B)
$$
are called \emph{supermodular}. A function is modular if and only if it is
both submodular and supermodular.

\smallskip

A general example of a modular function is a (possibly signed) discrete
measure on $X$. In the more general setting $X$ is infinite and the 
function is defined only on a sublattice of all subsets of $X$, e.g.
the Lebesgue measure on the measurable subsets of the unit interval.
Taking the outer measure instead, which is defined on all subsets of the
base set, we also get a suitable $f$, but it is only submodular, see, 
e.g., \cite{measure}. In another important example the underlying lattice is the
lattice of the vectors in the $n$-dimensional Euclidean space endowed with the
coordinatewise minimum and maximum as lattice operations. Here
the function $f:\R^n\to \R$ is submodular if
\begin{equation}\label{eq:4}
    f(\mathbf x) + f(\mathbf y) \ge f(\mathbf x \land \mathbf y)
         + f(\mathbf x \lor \mathbf y).
\end{equation}
Writing $\mathbf z=\mathbf x \land \mathbf y$, $\mathbf a=\mathbf x-\mathbf
z$, $\mathbf b=\mathbf y-\mathbf z$, both $\mathbf a$ and $\mathbf b$ have
non-negative coordinates, and (\ref{eq:4}) rewrites to
$$
    f(\mathbf z{+}\mathbf a)-f(\mathbf z) \ge f(\mathbf z{+}\mathbf b+\mathbf
a)-f(\mathbf z{+}\mathbf b).
$$
This formula is interpreted as the ``diminishing returns property''
\cite{osadeghi}: investing (adding) $\mathbf a$ at state $\mathbf z$ yields the
return $f(\mathbf z{+}\mathbf a)-f(\mathbf z)$. Investing the same amount
$\mathbf a$ after another investment $\mathbf b$ has been done yields a
smaller return. In this context supermodularity corresponds to
``accelerating returns'' \cite{Kurzweil2004}, where the same amount of
investment applied later yields larger returns.

\smallskip

The present paper deals exclusively with the case when $X$ is finite and
functions are defined on all subsets of $X$. The usual notation is used:
subsets of the base set $X$ are denoted by upper case letters such as $A, B,
K, L$, etc.; elements of $X$ by lower case letters, e.g. $i,j,k$. The $\cup$ sign
denoting the union of two subsets is frequently omitted as well as the curly
brackets around singletons. Thus, for example, $Ai$ denotes the subset
$A\cup\{i\}$. The collection of all subsets of $X$, including the empty set,
is denoted by $2^X$.

\smallskip

The following simple claim summarizes the basic properties of $\MOD(X)$, the
class of modular functions on a finite set $X$, see \cite{imset12}.

\begin{claim}\label{claim:1}\correcthref\begin{itemz}
\item[\upshape a)] Choose $z\in \R$ and $a_i\in\R$ for $i\in X$ arbitrarily.
Then $m:2^X\to\R$ defined by $m(A) = z+ \sum\{ a_i: i\in A\}$ for
$A\subseteq X$ is modular.
\item[\upshape b)] Every $m\in\MOD(X)$ is of this form. Consequently,
\item[\upshape c)] $\MOD(X)$ is an $|X|+1$-dimensional linear
space.\qed
\end{itemz}\end{claim}

Submodular functions remain submodular after adding (or subtracting) a
modular function. Also, a conic (that is, non-negative linear) combination
of submodular functions is again submodular. A collection $\mathcal G$ of
submodular functions over $X$ is a \emph{generator} if every submodular $f$
is the sum of a modular function and a conic combination of elements of
$\mathcal G$. A submodular function is \emph{extremal} if it is an element
of a minimal generator set. Extremal submodular functions are determined uniquely up to
a (positive) scaling factor and a modular shift, see Theorem \ref{thm:1}.
Since for a fixed finite $X$ there are only finitely many such extremal submodular
functions by Theorem \ref{thm:2}, it is possible, at least in theory, to list all
extremal submodular functions. This list also provides all extremal supermodular functions
as well since $f$ is submodular if and only if $-f$
is supermodular. Such
a list would carry invaluable information about the structure of submodular
functions. Several optimization techniques rely on finding an appropriate
extremal submodular function; with such a list, such problems would reduce to a simple search. Knowledge of
all extremal functions for the case $|X|=4$ was an essential ingredient in
investigating properties of conditional independence structures in
\cite{Matus-Studeny} and \cite{Studeny-revisited}. Extremal supermodular
functions for $|X|=5$ were reported in \cite{FiveVars}; these functions were
used in \cite{Csirmaz2020} to investigate and create new non-Shannon type
entropy inequalities for five random variables. 
In a more general setting, enumerating extremal rays is an essential
tool in objective space vector optimization problems \cite{ulus21}.

In our work attempting to generate all extremal submodular functions for
$|X|=6$
the main tools are methods from high-dimensional polyhedral
computations. For basic concepts and notions of this area, consult \cite{ziegler};
a comprehensive overview from the computational point of view is M.~Fukuda's
excellent monograph \cite{fukuda20}. A comparative study of different polyhedral
methods and implementations of different algorithms can be found in \cite{avis2016}. For notions and methods from linear algebra consult \cite{strang23}.

\section{The cone of submodular functions}\label{sec:1}

Let $n=|X|$ be the number of elements in the (finite) base set $X$. Any real function
defined on the subsets of $X$ can be represented by a $2^n$-dimensional
vector $r$ indexed by the subsets of $X$ as
$$
    r = \langle r_A: A \subseteq X \rangle,
$$
where $r_A$ is the value of the function at $A$. Both the function and the
vector notation will be used interchangeably.

\subsection{Submodular inequalities}\label{subsec:inequalities}

For disjoint subsets $A$, $B$, $K$ of $X$ with $A$ and $B$ not empty, let
$\delta(A,B|K)$ be the $2^n$-dimensional vector with four non-zero
coordinates: $+1$ at indices $AK$ and $BK$, and $-1$ at indices $K$ and
$ABK$. (The assumptions on $A$, $B$, $K$ ensure that these indices are
different.) The scalar product of $r$ and $\delta(A,B|K)$ is
$$
   r\cdot\delta(A,B|K)= r_{AK} + r_{BK} - r_K - r_{ABK}.
$$ 
A function represented by $r$ satisfies all inequalities in (\ref{eq:2}), that is, $r$ is submodular, if
and only if the scalar product $r\cdot\delta(A,B|K)$ is non-negative for
every choice of $A$, $B$ and $K$.
We will also use the bare triplet $(A,B|K)$ to mean the
\emph{inequality} expressing
$$f(AK)+f(BK) - f(K)-f(ABK) \ge 0$$
for some unspecified
function $f$, while $\delta(A,B|K)$ is the vector corresponding to, or
labeled by, this inequality.

The complete set of the inequalities $(A,B|K)$ is 
highly redundant. It can be illustrated based on the equality
\begin{equation}\label{eq:chain}
    \delta(A,B|K)+\delta(A,C|KB) = \delta(A,BC|K)
\end{equation}
known as the \emph{chain rule} in Information Theory \cite{Yeungbook}.
From this, if the inner products with the vectors on the left hand side are non-negative,
then so is with the one on the right hand side. Therefore the inequality
$(A,BC|K)$ is a consequence
of $(A,B|K)$ and $(A,C|KB)$. The unique minimal set of inequalities which
implies all others consists of the so-called elementary inequalities
(the terminology is from \cite{Studeny10}, see also \cite{imset12}). The inequality $(A,B|K)$ is
\emph{elementary} if both $A$ and $B$ are singletons. The matrix formed from the
row vectors corresponding to the elementary inequalities is denoted by $M^\sharp$:
$$
 M^\sharp = \big\{ \delta(i,j|K) : i,j\in X, K\subseteq X\setminus ij \big\}.
$$
Clearly, $M^\sharp$ has $2^n$ columns and ${n \choose 2}2^{n-2}$ rows.

\begin{claim}\label{claim:3}\correcthref
\begin{itemz}
\item[\upshape a)] The function represented by the vector $r$ is submodular if
and only if $M^\sharp\cdot r \ge 0$.
\item[\upshape b)] All elementary inequalities in $M^\sharp$ are necessary.
\item[\upshape c)] $M^\sharp$ has a rank of $2^n-n-1$.
\end{itemz}%
\end{claim}
\begin{proof}
a) To show that $M^\sharp\cdot r\ge 0$ implies $r\cdot\delta(A,B|K)\ge 0$ use
induction on $|A|+|B|$. If $|A|+|B|=2$, then $(A,B|K)$ is elementary, 
thus it is a row in $M^\sharp$. Otherwise, either $A$ or $B$ has at least two 
elements. Use the chain rule (\ref{eq:chain}) and induction.

\smallskip\noindent
b) To show that the row $\delta(i,j|K)$ in $M^\sharp$ cannot be omitted, it suffices
to exhibit a non-submodular function $g$ which satisfies all elementary
inequalities except this one. Let $|K|=k$ (clearly $0\le k\le n{-}2$), and
set $g(iK)=g(jK)=k$. For all other subsets $A\subseteq X$ define $g(A)=\min\{
|A|,k{+}1\}$. Then 
$$
   g\cdot\delta(i,j|K)= k+k-k-(k{+}1) = -1, 
$$
and it is easy to check that for the other elementary triplets the inner product 
$g\cdot\delta(i',j'|K')$ is either zero or plus one.

\smallskip\noindent
c) The kernel (zero set) of $M^\sharp$ is the set of modular functions as
$M^\sharp\cdot m=0$ if and only if $m$ is modular. By point c) of Claim \ref{claim:1} 
modular functions form an $n{+}1$-dimensional linear space. $M^\sharp$ has $2^n$
columns, thus the rank of $M^\sharp$ is $2^n-(n+1)$, as was claimed.
\end{proof}

\subsection{Standardization}\label{subsec:standard}

Two submodular functions are ``equivalent'' if their difference is modular.
This relation clearly splits the submodular functions into equivalence
classes. \emph{Standardization} is a method that assigns the same
representative to each element of such a class. For supermodular functions
three natural, theoretically motivated standardization methods are mentioned in
\cite[Chapter 5.1]{Studeny10}. For submodular functions, however, with a
focus on matroid theory, the following \emph{polymatroidal} standardization is
recommended. Consider the following linear space of functions on $2^X$:
$$
   \mathcal S_p(X) = \{ f: f(\emptyset)=0 \mbox{ and }
      \forall i\in X ~ f(X\setm i)=f(X) \,\}.
$$
The \emph{$p$-standardized} form of $f$ is the only element which is
both in the equivalence class $f+\MOD(X)$ and the linear space $\mathcal S_p(X)$.

As defined in \cite{Csirmaz2020} or in \cite{M2016}, \emph{tight
polymatroids} on $X$ are those submodular functions which are additionally
\begin{itemz}
\item[~] pointed:  $f(\emptyset)=0$,
\item[~] monotone: $A\subseteq B$ implies $f(A)\le f(B)$, and
\item[~] tight at the top: $f(X)=f(X\setm i)$ for all $i\in X$.
\end{itemz}
Clearly a $p$-standardized function $g$ is both pointed and tight.
Since monotonicity is a consequence of submodularity and these two
properties, $g$ is also monotone, thus it is a tight polymatroid.
The difference of two different tight polymatroids is never modular, which proves

\begin{claim}\label{claim:2}
The class of $p$-standardized submodular functions is the class of tight
polymatroids.\qed
\end{claim}

A consequence of Claim \ref{claim:2} is that $p$-standardized submodular
functions are automatically non-negative. One of our algorithms sketched as
Code \ref{code:4} uses this
fact as a quick preliminary check.

Other standardizations can be defined analogously by choosing different
$(n{+}1)$-dimensional linear subspaces which intersect each $f{+}\MOD(X)$ class
in a single element. Two of the subspaces suggested in \cite{Studeny16} are
the lower space
$$
   \mathcal S_\ell(X)=\{ f: f(A)=0 \mbox{ when $|A|\le 1$} \},
$$
and the upper space
$$
   \mathcal S_u(X) =\{ f: f(A)=0 \mbox{ when $|A| \ge n-1$}\},
$$
and there are many other possibilities. Using a different standardization
has no, or little, effect on the computational complexity of our algorithms
(however, it might affect the magnitude of the numbers to work with), thus
the choice of $\mathcal S_p$ is rather arbitrary, and was influenced mainly
by the authors' familiarity with polymatroids.

Theoretically, standardization is determined by a matrix $N$ with $2^n$
columns and $n+1$ rows such that the composite matrix ${M^\sharp\choose N }$
has full rank. $N$-standardized submodular functions are those
$2^n$-dimensional vectors for which both $M^\sharp r\ge 0$ and $N r=0$ hold.
Therefore, these vectors sit in the $2^n-(n{+}1)$-dimensional subspace
orthogonal to $N$. Using some alternate coordinate system of the
$2^n$-dimensional space with coordinates either in $N$ or orthogonal to $N$,
the overall dimension of the standardized functions is reduced from $2^n$ to
$2^n-(n{+}1)$.

\subsection{The cone of $p$-standardized functions}\label{subsec:std}

In case of lower or upper standardization, the subspace $\mathcal S_\ell$
(or $\mathcal S_u$, respectively) is spanned by the lowest (topmost) $n+1$
coordinates, thus the reduction simply discards these coordinates (and
replaces them by zeros). In case of $p$-standardization, the reduced
function $g$ is determined by the coordinates in the set
\begin{equation}\label{eq:R}
    \mathcal R = \{ A\subseteq X : A\ne \emptyset \mbox{ and } |A|\ne n{-}1
\}.
\end{equation}
Let $g$ be a vector with coordinates in $\mathcal R$, and expand it to a
$2^n$-dimensional vector $\tilde g$ by defining the values at the missing
places as $\tilde g(\emptyset)=0$ and $\tilde g(A)=g(X)$ for $|A|=n-1$.
Clearly, $\tilde g$ is in the linear space $\mathcal S_p$, thus $g$ is a
reduced $p$-standardized submodular function if and only if $M^\sharp \tilde
g \ge 0$. This product, however, can be computed directly from $g$. Let $M$
be the matrix obtained from $M^\sharp$ by deleting the column corresponding
to the empty set (as $\tilde g$ is zero at $\emptyset$), and replacing the
$n+1$ columns corresponding to subsets $A$ with $|A|\ge n-1$ by their sum
(as $\tilde g$ has the same value at these indices). In other words,
$M=M^\sharp S_p$, where $S_p$ is the matrix
$$
\def\zerobox#1{\hbox to 0pt{\hspace{20pt}\footnotesize$\Leftarrow$ #1\hss}}%
S_p =    \left(\begin{array}{cc}
            0 & 0\zerobox{1 row}\\
            \mathbf I & 0\zerobox{$2^n-(n{+}2)$ rows}\\
             0 & \mathbf 1\zerobox{$n+1$ rows}
          \end{array}\right)
          \hspace{60pt}
$$
Here $\mathbf I$ is the unit matrix and $\mathbf 1$ is a column of $n+1$
ones. Since $\tilde g=S_p g$, we have $M g = M^\sharp S_p g = M^\sharp 
\tilde g$. Let us define 
\begin{equation}\label{eq:C}
\mathcal C = \{ g: M g\ge 0\},
\end{equation}
where $g$ is a $2^n-(n{+}1)$-dimensional vector with indices from $\mathcal
R$. Clearly, $\mathcal C$ is the set of reduced and $p$-standardized
submodular functions, therefore every submodular function on $X$ is the sum
of a modular function and the expansion of an element from $\mathcal C$.

\begin{claim}\label{claim:C}
$\mathcal C$ is a full-dimensional pointed polyhedral cone.
\end{claim}
\begin{proof}
$\mathcal C$ is the intersection of ${n\choose 2}2^{n-2}$ many half-spaces
(the number or rows in $M$), thus it is polyhedral. All of these
halfspaces contain the origin, thus $\mathcal C$ is the union of rays
(half-lines) starting from the origin. By Claim \ref{claim:2}, $\mathcal C$
is a subset of the non-negative orthant (all coordinates of a $p$-standardized
submodular function are non-negative), thus $\mathcal C$ does not contain a
full line; and, in particular, $\mathcal C \cap -\mathcal C=\{0\}$. To prove that
$\mathcal C$ is full-dimensional, it is enough to show that it contains $2^n-n-1$ many linearly 
independent vectors. Choose $J\subseteq X$ with at least two
elements (observe that there are $2^n-n-1$ many such subsets). For each of
them, consider the function 
\begin{equation}\label{eq:fj}
     f_J(A) = \left\{\begin{array}{rl}
               1 & \mbox{if $J\cap A \ne \emptyset$,}\\
               0 & \mbox{otherwise. }
              \end{array}\right.
\end{equation}
It is easy to check that $f_J$ is both submodular and $p$-standardized
(since $|J|\ge 2$, $f_J(A)=1$ when $|A|\ge n-1$). Linear independence of the
vectors $f_J$ can be checked directly (by induction on the number of
elements in $X$), or by observing that their M\"obius transform gives all
unit vectors. For details consult \cite[Lemma 3]{fmtwocon}.
\end{proof}

An immediate consequence of this claim is that the matrix $M$ in
(\ref{eq:C}) has full rank.

\smallskip

Let us recall some further terminology of polyhedral geometry from
\cite{fukuda20} and \cite{ziegler}. A \emph{face} of the cone $\mathcal C$ is
a subset $F$ of $\mathcal C$ with the property that if a positive convex
combination of points in $\mathcal C$ falls into $F$, then the starting
points are also in $F$. Both $\mathcal C$ and the empty set are
faces---they are the trivial ones---, and the only single-element face is the
\emph{vertex} of the cone, which in this case is the origin. Proper faces
are just the intersection of $\mathcal C$ and a supporting hyperplane. The
dimension of a face $F$ is its affine dimension; for pointed cones it is the
maximal number of linearly independent vectors in $F$. One-dimensional faces
are \emph{edges} or \emph{extremal rays}; and a \emph{facet} is a face of
codimension $1$. Each proper face $F$ is the intersection of all facets
containing $F$. Facets of $\mathcal C$ lie on and are identified by the hyperplanes defined by
the rows of the matrix $M$---this follows from point b) of Claim
\ref{claim:3} which says that none of these rows is redundant.

\begin{theorem}\label{thm:1}
Extremal submodular functions on $X$ are unique up to a positive scaling and
a shift by a modular function.
\end{theorem}
\begin{proof}
For a submodular function $f$ let its $p$-standardized and reduced image be
$f^o$. Observe that the map $f\mapsto f^o$ is linear, and the
image is the complete cone $\mathcal C$. Linearity implies that conic
combination is preserved. Thus if $f^o$ is not on an extremal ray of
$\mathcal C$, then $f$ is not extremal. At the same time,
supermodular functions whose image is on some extremal ray of $\mathcal C$
form a generator: every other supermodular function is a positive conic
combination of them. As supermodular functions whose image is the same $g\in\mathcal
C$ differ by a modular shift only, the claim of the theorem follows.
\end{proof}

\begin{theorem}\label{thm:2}
For a finite $X$ there are finitely many extremal submodular functions.
\end{theorem}
\begin{proof}
According to Theorem \ref{thm:1} extremal submodular functions are, up to
scaling and modular shift, in a one-to-one correspondence with the extremal
rays of the cone $C$. Every extremal ray of $\mathcal C$ is the intersection
of the facets it is a subset of. Since $\mathcal C$ has ${n \choose
2}2^{n-2}$ many facets, it has finitely many extremal rays, which proves the
statement.
\end{proof}
Theorem \ref{thm:2} gives a trivial upper bound on the number of extremal
submodular functions. This is $2^{80}\approx 10^{24}$ for $n=5$, while
the actual value is around $10^5$, see Table \ref{table:1}. For $n=6$
we also expect a huge gap between the bound $2^{240}\approx 10^{72}$ given
by this theorem and the actual value.

\subsection{Symmetries}\label{subsec:symmetry}

Submodular functions have many symmetries. The most notable ones are induced by 
permutations of the base set $X$. Let $\pi$ be such a permutation of $X$, which we
extend to functions on $2^X$ by 
$$
    (\pi f)(A) = f(\pi A),
$$
where $\pi A$ is the image of $A\subseteq X$ under $\pi$. Trivially, $f$ is submodular
(supermodular or modular) iff $\pi f$ is such. There is a less known
symmetry of submodular functions called \emph{reflection} \cite{Studeny10} 
defined as
$$
    f^\c(A) = f(X\setm A),
$$
arising from the permutation $A\mapsto X\setm A$ of the subsets of $X$.
Claim \ref{claim:5} essentially says that these are the only symmetries
induced by permutations of subsets of $X$ which map the complete
set of inequalities in (\ref{eq:2}) onto itself.

\begin{claim}\label{claim:5}
Suppose a permutation of the subsets of $X$ induces a permutation of the inequalities in
(\ref{eq:2}). Then it is equivalent to an optional reflection followed by a symmetry
induced by a permutation of $X$.
\end{claim}
\begin{proof}
Let a permutation of the subsets of $X$ be defined by the bijection $\phi$,
and rewrite the inequalities in (\ref{eq:2}) by mapping each subset $A$ to $\phi(A)$.
For each subset, count how many times it occurs on the right-hand side of the inequalities.
There will be two with the maximal count $(3^n{+}1)/2-2^n$; which means they must be
$\emptyset$ and $X$, but also $\phi(\emptyset)$ and $\phi(X)$ in some order.
If $\phi(\emptyset)=X$, apply a reflection, and assume $\phi(\emptyset)=\emptyset$ going forward.
On the right hand side of the inequalities every subset
occurs next to the empty set, except for singletons.
This means $\phi$ is a bijection on singletons, which determines the permutation of the base set
that will give rise to $\phi$.
Each pair of singletons occurs exactly once on the left hand side, with the corresponding
right-hand side containing the empty set and the union of the two singletons.
This means that $\phi$ for two-element subsets is uniquely determined by its values on singletons.
The same logic for larger subsets yields that $\phi$ is also determined on all remaining subsets of $X$.
\end{proof}

A permutational symmetry $\pi$ automatically maps to a symmetry of the
cone $\mathcal C$. It is because $\pi$ induces a permutation of the coordinate
set $\mathcal R$ defined in (\ref{eq:R}), which, in turn, induces a
permutation of the rows of the matrix $M$ so that for each row
$\delta$ of $M$ and each vector $g$ with coordinates in
$\mathcal R$ we have
$$
    \delta\cdot g = \pi(\delta)\cdot \pi(g).
$$
In particular, $M g\ge 0$ if and only if $M \pi(g)\ge
0$, thus $\mathcal C = \pi(\mathcal C)$.

The case of reflection is more subtle as $f^\c$ is not
necessarily $p$-standardized. To get the reflected pair $\sigma f$ of
the $p$-standardized submodular function $f$ one has to shift $f^\c$ by the
modular function
$$
    m(A) = -f(X) + {\textstyle\sum}\{ f(i): i \in A\,\}
$$
(see part a) of Claim \ref{claim:1}) to get the $p$-standardized submodular
function
\begin{equation}\label{eq:dual}
    (\sigma f)(A) = f(X\setm A)-f(X) + {\textstyle\sum}\{ f(i): i \in A\,\}.
\end{equation}
In matroid and polymatroid parlance \cite{oxley} $\sigma f$ is called the
\emph{dual} of $f$. Reflection also induces a permutation on the rows of
$M^\sharp$ (and then on rows of $M$) as the reflected image of the
elementary inequality $(i,j|K)$ is the elementary inequality $(i,j|X\setm 
ijK)$. This permutation of inequalities, denoted again by $\sigma$, also satisfies the
exchange property
$$
    \delta \cdot g = \sigma(\delta)\cdot\sigma(g).
$$
Consequently $\mathcal C=\sigma(\mathcal C)$, thus $\sigma$ is another symmetry
of the cone $\mathcal C$.

\smallskip

All of these symmetries induce a linear transformation on the coordinates, consequently
they preserve extremality. Symmetric images of a fixed extremal ray form an
\emph{orbit}. As the symmetry transformations can be computed trivially, it suffices to
compute one ray from each orbit. Since reflection is idempotent, there are
$2n!$ many symmetries. This means a significant reduction as we expect the
majority of the orbits to contain $2n!$ many different
rays. Our computations confirm that this is indeed the case, see Figure \ref{fig:6}.

\subsection{Extending the base set}\label{subsec:extending}

Suppose the base set $X$ is extended by a new element $y$ to the larger set
$Y=Xy$. We extend a function $f$ defined on the subsets of $X$ to a
function $f^\star$ defined on all subsets of $Y$ by
$$
   f^\star(yA)=f^\star(A) = f(A) ~~\mbox{for $A\subseteq X$},
$$
in particular, $f^\star(y)=f(\emptyset)$.
Observe that if $f$ is a $p$-standardized submodular function, then so is
$f^\star$. Consequently the map $g \mapsto g^\star$ embeds the cone
$\mathcal C_X$ of reduced, $p$-standardized submodular functions over $X$
into the cone $\mathcal C_Y$.

\begin{claim}\label{claim:exex}
Suppose $r$ is an extremal ray in $\mathcal C_X$. Then $r^\star$ is an
extremal ray in $\mathcal C_Y$.
\end{claim}
\begin{proof}
Suppose $r^\star$ is a positive conic combination of rays $s_i$ in $\mathcal
C_Y$, that is, $r^\star = \sum_i \lambda_i s_i$ where all $\lambda_i$ are
positive. We must show that $s_i$ is a multiple of $r^\star$. Since
$r^\star(y)=0$, we have $s_i(y)=0$. By Claim \ref{claim:2} the submodular
function represented by $s_i$ is monotone and pointed, thus for every 
$A\subseteq X$
$$ 
    0 \le s_i(Ay)-s_i(A)\le s_i(y)-s_i(\emptyset) = 0.
$$
This means that there is an $r_i$ in $\mathcal C_X$ such that $s_i=r_i^\star$,
and then $r^\star=\sum_i \lambda_i r_i^\star$. Since $r\mapsto
r^\star$ is a linear map, it implies $r=\sum_i \lambda_i r_i$. By assumption
$r$ is extremal, thus $r_i$ is a multiple of $r$, and then $s_i=r_i^\star$
is a multiple of $r^\star$, as required.
\end{proof}

The embedding $g\mapsto g^\star$ of $\mathcal C_X$ into $\mathcal C_Y$ also
preserves the symmetries of $\mathcal C_X$. Denote the duality symmetry on
$X$ and $Y$ by $\sigma_X$ and $\sigma_Y$, respectively. Similarly, extend
the permutation $\pi_X$ on $X$ to a permutation $\pi_Y$ on $Y$ by keeping
its effects on $X$ and stipulating $\pi_Y(y)=y$.

\begin{claim}\label{claim:exsym}
For any $g\in\mathcal C_X$, $(\sigma_X g)^\star = \sigma_Y g^\star$, and 
$(\pi_X g)^\star = \pi_Y g^\star$.
\end{claim}
\begin{proof}
Clearly we have $g^\star(Y)=g^\star(X)=g(X)$, and also for all subsets 
$A\subseteq Y$, $g^\star(Y\setm A)=g(X\setm A)$ and 
$$
   \textstyle \sum\{g^\star(i):i\in A\} = \sum\{g(i):i\in A\setm y\}.
$$
From the above and (\ref{eq:dual}), for $A\subseteq X$ we get
\begin{align*}
    & (\sigma_Y g^\star)(A) = (\sigma_Y g^\star)(Ay) ={}\\
    &~~{}=g(X\setm A)+\sum\{g(i):i\in A\} = (\sigma_X g)(A),
\end{align*}
as was claimed. The second statement can be checked analogously.
\end{proof}

According to Claim \ref{claim:exex}, extremal submodular functions on an
$n+1$-element base contain the extremal submodular functions on a base with
$n$ elements---and then, by induction, all extremal submodular functions
on smaller bases. Claim \ref{claim:exsym} strengthens this result: if,
instead, we consider orbits, then every orbit on a smaller base occurs
exactly once as an orbit on a larger base. For example, taking those orbits 
for $n=6$
which have a representative vanishing on some
singleton, we get exactly the 672 orbits of $n=5$, see Table \ref{table:1}.

\section{Methods for generating extremal rays}\label{sec:eer}

This section gives an overview of some methods from polyhedral geometry
\cite{fukuda20,joswig13} which were used in our computations. Recall that the
problem of finding all extremal submodular functions was reduced to the task
of finding all extremal rays of a pointed, full-dimensional, polyhedral cone.
For this exposition the dimension of the cone is denoted by $d$, and the cone is the
intersection of $m\ge d$ irredundant positive halfspaces specified as
$$
    \mathcal C = \{ x\in \R^d: Mx \ge 0 \},
$$
where the matrix $M$ has $d$ columns, $m$ rows, and has rank $d$. Irredundant
means that no proper subset of the rows of $M$ define the same cone. In other
words, the hyperplanes 
determined by the rows of $M$ are the bounding facets (that is,
$d-1$-dimensional faces) of $\mathcal C$.

The \emph{ray} generated by a non-zero point (or vector) $x\in\R^d$ is
$r=\{\lambda x: \lambda\ge 0\}$. By abusing the notation a ray is identified
with one (or any) of its generating points. Thus we write $r\in\mathcal C$
to mean that all points of $r$ are in $\mathcal C$, and $Mr\ge 0$ to mean
that this relation holds for some non-zero (and then all) points of $r$.

The ray $r$ is \emph{extremal} in $\mathcal C$ if it is a one-dimensional
face of $\mathcal C$. Recalling the definition of polyhedral faces
\cite{ziegler}, $r$ is extremal iff it is not a strictly positive conic
combination of other rays of $\mathcal C$.

The \emph{support} of the ray $r\in\mathcal C$, denoted by $\mofr$, is the set
of those bounding hyperplanes of $\mathcal C$ the ray is on, namely $\mofr =
\{ a\in M: a\cdot r=0 \}$. Worded differently, $\mofr$ is the set of active
constraints for $r$, or the set of facets of $\mathcal C$ containing $r$.
The following claim follows easily from the definitions, see also
\cite{fukuda20,ziegler}.

\begin{claim}\label{claim:rank1}
The ray $r\in\mathcal C$ is extremal if and only if $\mofr$ has rank $d-1$, 
and then $r$ is the unique one-dimensional solution of the homogeneous
system $(\mofr)\, x = 0$. \qed
\end{claim}

The number of facets $r$ is on is called its \emph{weight} and is denoted by
$w(r)$; this number is called \emph{incidence number} in \cite{sikiric07}.
Clearly, $w(r)$ is the number of rows in $\mofr$. If $\mofr$ has rank $d-1$,
then it must have at least $d-1$ rows. Consequently extremal rays have
weight $w(r)\ge d-1$.

\subsection{The Double Description Method}\label{subsec:DD}

A polyhedral cone is determined both by its bounding hyperplanes arranged
into the matrix $M$, and by the set of its extremal rays arranged into the
matrix $R$. Several algorithms and implementations exist for enumerating the
rows of $R$ given the matrix $M$, see \cite{howgood,joswig13}.
Interestingly, the converse problem, namely enumerating the bounding
hyperplanes given the set of extremal rays, is a completely equivalent
problem. This is because if the rows of $R$ are interpreted as hyperplanes,
then the cone they generate has exactly $M$ as the set of its extremal rays,
see, e.g., \cite{fukuda20,ziegler}.

The ray enumeration algorithm that fits the highly degenerate case of
submodular functions best is the iterative \emph{Double Description Method},
abbreviated as DD. It was described first by Motzkin et al.~\cite{motzkin};
for a recent overview see \cite{fukuda-prodon}. DD is a variant of the
Fourier-Motzkin elimination, or Chernikova's algorithm, and is sketched as
Code \ref{code:1}. It starts with a subset $M_0$ of $d$ inequalities from
$M$ which defines a full-dimensional cone with $d$ extremal rays. In an
intermediate step we have a cone determined by the inequalities (or facets)
in $M_i\subseteq M$, and also we have the complete list $R_i$ of its
extremal rays. This way each intermediate cone has double description,
giving the name of the method. To obtain the next cone, a new inequality $a$
is added to $M_i$, that is, part of the cone is cut off by the hyperplane
corresponding to $a\ge 0$. This hyperplane partitions the old rays in $R_i$
into positive, negative, and zero ones depending on whether the ray is on
the positive side of the hyperplane, on the negative side, or it is on the
hyperplane. Positive and zero rays remain extremal in the next cone;
negative rays are part of the removed segment, and so are not part of the
new cone. All additional extremal rays are on the cutting hyperplane. These
are the rays where the conic span of an \emph{adjacent pair} of a positive
and a negative ray intersects this hyperplane; for details, see
\cite[Section 8.1]{fukuda20}. The algorithm terminates when there are no
more inequalities to add.

\begin{pseudocode}{The Double Description Method\label{code:1}}
\State Compute the initial DD pair $(M_0,R_0)$; $i\gets 0$\label{c1-l1}
\While{there is $a\in M\setm M_i$}
  \State $M_{i+1} \gets M_i\cup\{a\}$\label{c1-l3}
  \State Split $R_i$ into positive / negative / zero rays\label{c1-l4}
  \State $R_{i+1}\gets{}$ positive and zero rays
  \ForEach{$r_1$ positive / $r_2$ negative ray}\label{c1-l6}
  \If{$r_1$ and $r_2$ are adjacent }\label{c1-l7}
    \State Compute the ray $r={}$conic$(r_1,r_2)\cap a$\label{c1-l8}
    \State $R_{i+1}\gets R_{i+1}\cup\{r\}$
  \EndIf
  \EndFor\label{c1-l11}
  \State $i\gets i+1$
\EndWhile
\end{pseudocode}

The induction step (lines \ref{c1-l3} to \ref{c1-l11} of the outlined DD
algorithm) was implemented as a stand-alone program working as a
\emph{pipe}. It reads the DD pair $(M_i,R_i)$ together with the new
inequality $a\in M$, and produces the next DD pair $(M_{i+1},R_{i+1})$ which
is amenable for the next iteration. The initial DD pair (line \ref{c1-l1})
is created by the controlling program using simple linear algebraic tools
\cite{strang23}. The overhead caused by the pipe arrangement (shipping the
data in and out of memory at each iteration) is more than compensated for by
simpler data structure, smaller memory requirement, and the ability to
allocate all necessary memory in a single request (thus avoiding memory
fragmentation). As a bonus, almost linear speed-up can be achieved by
distributing the computation made by the pipe program among several cores
or, preferably, over several machines. Figure \ref{fig:3} gives an
illustration of the general performance by plotting the total number of rays
against the speed of ray generation. (Actually, the figure is for the
equivalent problem of enumerating adjacent rays as discussed in Section
\ref{subsec:neighbor}.)

The crucial part of the DD algorithm is testing the adjacency of extremal
rays of the current cone (line \ref{c1-l7} in Code \ref{code:1}). This
ensures that only extremal rays of the next iteration are added to the
output. The following equivalent criteria for ray adjacency are from
\cite{fukuda-prodon}.
\begin{claim}[Fukuda--Prodon \cite{fukuda-prodon}]\label{claim:adjacency}
For two extremal rays $r_1$ and $r_2$ the following conditions are equivalent.
\begin{itemz}
\item[\upshape a)] $r_1$ and $r_2$ are adjacent;
\item[\upshape b)] the rank of $\mof{r_1}\cap \mof{r_2}$ is $d-2$ (algebraic test);
\item [\upshape c)] there is no extremal ray $r$ apart form $r_1$ and $r_2$
such that $\mofr\supseteq \mof{r_1}\cap \mof{r_2}$ (combinatorial test). \qed
\end{itemz}
\end{claim}
If we write $w(r_1,r_2)$ to mean the number of rows in $\mof{r_1}\cap
\mof{r_2}$, then the algebraic test requires computing the rank of a matrix
with $d$ columns and $w(r_1,r_2)$ rows. The naive implementation requires
$O(d^2w(r_1,r_2))$ arithmetic operations, independently of the number of
extremal rays. As this computation is sensitive to numerical errors, the
rank is frequently computed using arbitrary precision arithmetic, which can
be very slow even for small values of $d$. Conversely, the combinatorial
test runs in time proportional to the number of extremal rays, potentially a
very large number. However, this test can be implemented using fast bit
operations on the ray--facet incidence matrix, and is not prone to numerical
errors. Both tests can be sped up using a simple consequence of the
algebraic test: $r_1$ and $r_2$ are definitely \emph{not} adjacent if
$w(r_1,r_2)$ is smaller than $d-2$. This simple condition should be checked
before delving into any of the two more demanding tests.

\subsection{Enumerating neighbors}\label{subsec:neighbor}

Another avenue to finding extremal rays of a cone is to find extremal rays
adjacent to a known extremal ray \cite{bremner09,Plos17,rehn10,sikiric07},
also known as \emph{Adjacency Decomposition}. Since extremal rays of a cone
are connected with respect to the adjacency relation \cite[Theorem
3.14]{ziegler}, starting from any extremal ray of the cone and determining
its neighbors, then the neighbors of these rays, and so on, will eventually
generate all extremal rays.

For the details let $r$ be such a fixed extremal ray of the cone $\mathcal
C=\{x\in\R^d: Mx\ge 0\}$. Enumerating the neighbors of $r$ can be done by
solving a generic ray enumeration problem in dimension $d-1$. To show that
this is the case, let $M'=\mofr$ be the set of facets $r$ is on, and let
$z\in M\setm M'$ be any of the remaining facets of $\mathcal C$. Clearly
$M'r = 0$ and $z\cdot r > 0$ as $z$ is a bounding facet of $\mathcal C$, and
therefore $r$ is on the positive side of $z$. Consider the following
$(d-1)$-dimensional cone $\mathcal C'$ embedded into the $d$-dimensional
space $\R^d$:
$$
   \mathcal C' = \{ x\in \R^d: M'x\ge 0 \mbox{ and } z\cdot x=0 \}.
$$
Since $r$ is extremal, $M'$ has rank $d-1$. Since $M'r=0$ and $z\cdot r\ne
0$, $z$ is not in the linear span of the rows of $M'$, therefore the
composite matrix ${M'\choose z}$ has rank $d$, and then $\mathcal C'$ is
$(d-1)$-dimensional, as was claimed. The analog of Claim \ref{claim:rank1}
remains valid: a ray $s\in\mathcal C'$ is extremal if and only if
$\matrixrestr{M'}{s}$ has rank $d-2$. Similarly, both tests for adjacency of
extremal rays stated in Claim \ref{claim:adjacency} remain true for
$\mathcal C'$ if $d-2$ is replaced by $d-3$. Therefore the Double
Description Method with some minor modifications can be used to enumerate
the extremal rays of the cone $\mathcal C'$.

\begin{claim}
Extremal rays of $\mathcal C'$ and the neighboring rays of $r$ are in a
one-to-one correspondence.
\end{claim}
\begin{proof}
We begin by describing the mapping from rays adjacent to $r$ to extremal
rays in $\mathcal C'$. Let $r_1$ be an extremal ray of $\mathcal C$ adjacent
to $r$. Let $s= r_1-\lambda r$ where $\lambda$ is chosen so that $z\cdot
s=0$. (Since $z\cdot r\ne 0$, such a $\lambda$ exists.) We show that $s$ is
an extremal ray of $\mathcal C'$. First, $s\in\mathcal C'$ as $z\cdot s=0$
and
$$
    M's = M'\cdot(r_1-\lambda r) = M' r_1 \ge 0
$$
as $r_1$ is a ray in $\mathcal C$. Second,
$\matrixrestr{M'}{s}=\matrixrestr{M'}{r_1}=\mof{r_1}\cap \mofr$. Since $r$
and $r_1$ are adjacent, the rank of this matrix is $d-2$, which proves that
$s$ is extremal by Claim \ref{claim:adjacency}. It is easy to see that no
two neighbors of $r$ are mapped to the same ray of $\mathcal C'$.

In the other direction, let $s\in\mathcal C'$ be extremal, and consider the
the ray $r'=s+\mu r$ for some real number $\mu$. Clearly,
$\mpofrp=\matrixrestr{M'}{s}$, and this matrix has rank $d-2$. If $a$ is a
row in $M'$, then $a\cdot r' = a\cdot s \ge 0$. If $a$ is in $M$ but not in
$M'$, then $a\cdot r$ is positive, and
$$
   a\cdot r' = a\cdot(s+\mu r) = a\cdot s + \mu (a\cdot r) \ge 0
$$
if and only if $\mu\ge -(a\cdot s)/(a\cdot r)$. Choosing $\mu$ as the
minimum of these values, $r'$ will be a ray in $\mathcal C$. Moreover,
$\mofrp$ will contain the row $a\notin M'$ where this minimum is taken:
$\mofrp \supseteq \mpofrp\cup\{a\}$. Since $\mpofrp$ has rank $d-2$ ($r'$ is
extremal in $\mathcal C'$), $\mofrp$ has rank $d-1$, proving that $r'$ is
extremal in $C$. Finally, $\mofr\cap \mofrp = \mpofrp$ has rank $d-2$, thus
$r$ and $r'$ are adjacent by Claim \ref{claim:adjacency}. This construction
also indicates how neighbors of $r$ can be generated from the extremal rays
of $\mathcal C'$.
\end{proof}

The same reduction can be applied iteratively to the reduced cone $\mathcal
C'$. Suppose, in general, that a $k$-dimensional cone $\mathcal C$ embedded
in the $d$-dimensional space is defined as
$$
    \mathcal C = \{ x\in\R^d: Mx \ge 0, \mbox{ and } Ax=0 \,\}
$$
where $M$ has $m$ rows, $d$ columns, and rank $k$; $A$ has $d-k$ rows, $d$
columns, and rank $d-k$; and the composite matrix $M\choose A$ has rank $d$.
Let $r$ be an extremal ray of $\mathcal C$. Enumerating the neighbors of $r$
can be reduced to enumerating the extremal rays of the $k-1$-dimensional
cone $\mathcal C'$ defined by the matrix pair $\mofr$ and $A\choose a$ where
the row $a\in M$ is not in $\mofr$.

The reduction stops when the complexity of the cone $C$ defined by the
matrix pair $(M,A)$ becomes manageable thus its extremal rays can be
enumerated directly. For this purpose one can use the $d$-dimensional DD
method directly with some minor modifications as discussed above. Another
possibility is to make $\mathcal C$ full dimensional first by projecting it
to the $k$-dimensional space
$$
    \pi\mathcal C = \{ y \in\R^k: (MP) y \ge 0 \,\}
$$
where $P$ is a generator of the nullspace of $A$ (see
\cite{strang23} for details). Extremal rays of $\mathcal
C$ can be recovered by applying $P$ to the extremal rays of $\pi\mathcal C$.
Actually, such a projection was used for the $p$-standardized
submodular functions in Section \ref{subsec:std}.

\subsection{The kernel method}

When applying the DD method directly to enumerate all extremal rays of the
cone
$$
    \mathcal C = \{ x\in\R^d: Mx \ge 0, \mbox{ and } Ax=0 \,\},
$$
computations simplify considerably when each row of $M$ contains exactly one
non-zero entry. Using homogeneity and an optional coordinate sign change,
that entry can be assumed to be $1$. When the DD method adds the new row
$a\in M$, old rays are put into positive, negative, and zero parts depending
on the sign of the inner product $a\cdot r$, see line \ref{c1-l4} of Code
\ref{code:1}. Similarly, inner products $a\cdot r_1$ and $a\cdot r_2$ are
used in line \ref{c1-l8} when computing that conic combination of rays $r_1$
and $r_2$ which lies on the hyperplane determined by the row $a$. If the
only non-zero coordinate in $a$ is $a[i]=1$, then computing these inner
products reduces to the trivial task of taking the $i$-th coordinate.
Additionally, ray coordinates where the corresponding row in $M$ contains
zeros only need not to be stored at all; and among the processed coordinates
(determined by those rows of $M$ which were previously handled by DD) it
suffices to store only one bit indicating whether that coordinate of the ray
is zero or not. When DD finishes, these bits determine the matrix $\mofr$,
which allows (re)computing the coordinates of the ray $r$, postponing large
part of high-precision computations to the final stage.

This variant of the DD method is called \emph{kernel} or \emph{null-space
method} \cite{terzer09}. While the kernel method uses very specially defined
cones, it is universal, as the standard ray enumeration problem for the cone
$$
    \mathcal C = \{ x\in \R^d: Mx \ge 0 \,\}
$$
can be transformed into the required special form in the following way.
Suppose $M$ has $m$ rows. Let $I_m$
be the $m\times m$ unit matrix and $y\in\R^m$ be new \emph{slack} variables.
The first $d$ coordinates of the extremal rays of the cone
$$
    \{ (x,y) \in \R^{d+m}: y\ge 0, Mx + I_my = 0 \,\}
$$
are just the extremal rays of $\mathcal C$. The software package POLCO uses
the kernel method for ray enumeration. For a detailed description and
theoretical background of the package, see \cite{terzer09}.

\section{Improving the combinatorial test}\label{sec:bitmap}

In the DD method we need to check the adjacency of each pair of extremal
rays taken from the two sides of the new facet, that is, where one ray comes
from the positive class, and the other ray from the negative; see line
\ref{c1-l7} in Code \ref{code:1}. The majority of these pairs are expected
not to be adjacent, in which case the pair is skipped, thus this test needs
to be as fast as possible. This Section discusses the details and potential
improvements of the combinatorial test described in Claim
\ref{claim:adjacency}.

Let $M$ be the set of facets (inequalities), and $R$ be the set of extremal
rays of the cone $\mathcal C$. The ray--facet \emph{incidence matrix} has
one row for each facet and one column for each ray, and for $a\in M$ and
$r\in R$ the entry at position $(a,r)$ is 1 if the scalar product $a\cdot r$
is zero (the ray $r$ lies on the facet $a$, they are incident), and 0
otherwise. For each $a\in M$ let $\hat a$ be the 0--1 string formed from the
entries in row $a$ of this incidence matrix, and, similarly, $\hat r$ be the
0--1 string formed from the entries in column $r$. Clearly, $\mofr$ contains
those rows of $M$ where $\hat r$ is $1$. Denoting the number of 1's in $\hat
r$ by $|\hat r|$, we have $|\hat r|\ge d-1$ for every $r\in R$. This is
because by Claim \ref{claim:rank1} $\mofr$ has rank $d-1$, thus it has at
least $d-1$ rows. Similarly, the number of rows in both $\mof{r_1}$ and
$\mof{r_2}$, which was denoted by $w(r_1,r_2)$ above, is $|\hat r_1\cap\hat
r_2|$, where the intersection of two strings is understood to be taking the
minimal value at each position.

\begin{pseudocode}{Combinatorial adjacency test of rays $r_1$ and $r_2$\label{code:2}}
 \State \textbf{if} $|\hat r_1\cap\hat r_2|<d-2$ \textbf{then}
     return \textit{no}\label{c2-l1}
 \State $b\gets{}$all 1 bitstring of length $|R|$;\\
     \hbox to 12pt{\hfil} set $b[r_1]\gets 0$ and $b[r_2]\gets 0$\label{c2-l2}
 \ForEach{ $a\in \hat r_1 \cap \hat r_2$}\label{c2-l3}
    \State $b\gets b\cap \hat a$
    \State \textbf{if} $b=0$ \textbf{then} return \textit{yes}
 \EndFor\label{c2-l6}
 \State return \textit{no}
\end{pseudocode}

Code \ref{code:2} outlines the combinatorial test. Line \ref{c2-l1} executes
the quick precheck $w(r_1,r_2)\ge d-2$ coming from the algebraic test; it
filters out many of the cases. Line \ref{c2-l2} initializes the bitstring
$b$, and the loop in lines \ref{c2-l3}--\ref{c2-l6} computes those bits in
$b$ whose index $r$ satisfies $\mofr\supseteq \mof{r_1}\cap \mof{r_2}$ with
the exception of $r_1$ and $r_2$ whose indices are cleared in the
initialization. Finally, $r_1$ and $r_2$ are adjacent if no other ray
remains in $b$, that is, $b$ becomes the all zero string. This condition is
checked in the loop, skipping the remaining computation whenever possible.

\smallskip

N.~Zolotykh made the observation \cite{zolotykh19} that if the ray $r$
violates the combinatorial test, that is,
$$
    \mofr \supseteq \mof{r_1}\cap \mof{r_2},
$$
then both $\mofr\cap \mof{r_1}$ and $\mofr\cap \mof{r_2}$ contain the
intersection $\mof{r_1}\cap \mof{r_2}$ as a subset. In particular, if $|\hat
r_1\cap \hat r_2|\ge d-2$ (this pair survived the quick test), then both
$|\hat r\cap \hat r_1|\ge d-2$ and $|\hat r\cap\hat r_2|\ge d-2$. Thus it
suffices to restrict the search in the loop at \ref{c2-l3}--\ref{c2-l6} of
Code \ref{code:2} to such rays. So for a ray $r_i$ let $\hat g(r_i)$ (for
\emph{graph test}, a terminology used in \cite{zolotykh19}) be the bitstring
indexed by the rays such that at index $r$ this string has 1 if $r$ and
$r_i$ differ and $|\hat r\cap \hat r_i|\ge d-2$, and has 0 otherwise.
Instead of line \ref{c2-l2} in Code \ref{code:2}, $b$ can be initialized to
$b\gets \hat g(r_1)\cap\hat g(r_2)$. Since $b$ starts with fewer bits set,
the loop at \ref{c2-l3}--\ref{c2-l6} is expected to finish earlier.

Computing the strings $\hat g(r_1)$ and $\hat g(r_2)$ for all positive and
negative rays takes time, and it is not clear that this overhead is
compensated for by the improved performance. The main drawback, however, is
the significant memory needed to store these bitstrings. The best compromise
to retain some of the advantages of the graph test without extensive memory
requirement seems to be computing $\hat g(r_1)$ ``on the fly'' for positive
rays, also suggested in \cite{zolotykh19}. Our contribution is to complement
this by not using $\hat g(r_2)$ for negative rays at all. Performing the
adjacency tests in an appropriate order every $\hat g(r_1)$ needs to be
computed only once and the string may occupy the same memory location. This
version, dubbed as \emph{\halfg-graph test}, is sketched as Code
\ref{code:3}.
\begin{pseudocode}{\halfg-graph adjacency test of rays $r_1$ and $r_2$\label{code:3}}
 \If{$\hat g(r_1)$ is not defined}
   \ForEach{$r\in R$}\label{c3-l2}
     \State $\hat g(r_1)[r] \gets (r\ne r_1$ and $|\hat r\cap \hat r_1|\ge d-2)$
   \EndFor\label{c3-l4}
 \EndIf
 \State \textbf{if} $\hat g(r_1)[r_2]=0$ \textbf{then}
        return \textit{no}\label{c3-l6}
 \State $b\gets \hat g(r_1)$; set $b[r_2] \gets 0$\label{c3-l7}
 \ForEach {$a\in \hat r_1 \cap \hat r_2$}\label{c3-l8}
    \State $b\gets b\cap \hat a$
    \State \textbf{if} $b=0$ \textbf{then} return \textit{yes}
 \EndFor\label{c3-l11}
 \State return \textit{no}
\end{pseudocode}
The loop in lines \ref{c2-l2}--\ref{c3-l4} prepares the bitstring $\hat
g(r_1)$; this will be done only once for every positive ray $r_1$ under an
appropriate scheduling. The quick check in line \ref{c3-l6} reuses part of
this precomputation: the result of the condition $|\hat r_1\cap \hat
r_2|<d-2$ is simply looked up in the bitstring $\hat g(r_1)$. Line
\ref{c3-l7} initializes $b$ and the loop at lines \ref{c3-l8}--\ref{c3-l11}
searches for a ray $r$ such that $\mofr$ extends $\mof{r_1}\cap \mof{r_2}$.
The search is done in (typically 64 bit) chunks rather than over the whole
string at once. When a chunk becomes empty, (that is, all zero), which we
expect to happen frequently, the inner loop is aborted, further improving
the performance.

\section{Enumerating extremal submodular functions for $n=6$}\label{sec:apply}

We are now ready to apply the tools and algorithms described above to
specific values of $n$. To recap, Section \ref{subsec:std} defined the cone
of $p$-standardized submodular functions over a base set with $n\ge 3$
elements. Extremal submodular functions, up to a shift by a modular
function, were identified with the extremal rays of a polyhedral cone
$\mathcal C_n$. Enumerating extremal submodular functions means enumerating
these extremal rays. The cone $\mathcal C_n$ sits in the
$d=2^n-(n{+}1)$-dimensional space and is defined by $m={n\choose 2}2^{n-2}$
homogeneous inequalities as
$$
    \mathcal C_n = \{ g\in \R^d: M_n\, g\ge 0\,\},
$$
see Section \ref{subsec:std}. Rows of $M_n$ are determined by the
$p$-standardized elementary inequalities as discussed in Section
\ref{subsec:inequalities}. Each non-zero entry in $M_n$ is either $+1$ or
$-1$, and each row contains two, three or four non-zero entries only, making
$M_n$ to be extremely sparse. As shown in Section \ref{subsec:symmetry},
$\mathcal C_n$ has $2n!$ symmetries. The $n!$ part comes from permuting the
base set, and the factor $2$ comes from reflection. Symmetric images of
extremal rays are also extremal, thus it suffices to have one representative
from each orbit, that is, symmetry class.

Based on the above numbers and on later calculations, Table \ref{table:1}
lists the dimension, number of inequalities, symmetries, total number of
extremal rays, number of orbits, and the weight range (see the comment after
Claim \ref{claim:rank1}) of cones for different values of $n$.

\begin{table}[!htb]%
\def\fs{\footnotesize}%
\caption{Dimension, facet number, symmetries, extremal rays, orbits,
and weight range}\label{table:1}%
\centering
\begin{tabular}{|c|rrrrrc|}
  \hline
  \rule[-4pt]{0pt}{13pt}\fs $n$ & 
  \multicolumn{1}{|c}{\fs $d$} & \fs $m$ &\fs symm & 
\multicolumn{1}{c}{\fs rays} & \multicolumn{1}{c}{\fs orbits } & \fs weight \\
  \hline   
   \rule{0pt}{10pt}3 &  4  &  6 & 12  &  5~~ & 2~~ & 3--4 \\[2pt]
   4 & 11  & 24 & 48  & 37~~ & 7~~ & 10--20 \\[2pt]
   5 & 26  & 80 & 240 & $117\;978$~~ & 672~~{}  & 25--72\\[2pt]
   6 & 57  & 240 & 1440 & ${>}3.9{\cdot}10^{11}$ & ${>}2.6{\cdot}10^8$ & 56--225\\[3pt]
   \hline
\end{tabular}
\end{table}

Section \ref{subsec:extending} defined a canonical embedding of $\mathcal
C_n$ into $\mathcal C_{n+1}$; this embedding preserves both extremal rays
(by Claim \ref{claim:exex}), and orbits (by Claim \ref{claim:exsym}). In
particular, extremal rays of $\mathcal C_n$ can be recovered from the
extremal rays of $\mathcal C_{n+1}$ by simply looking for rays which take
zero at a coordinate labeled by a fixed singleton. Referring to Table
\ref{table:1}, preserving the orbits means that among the 672 orbits of
$\mathcal C_5$ there are seven which are the images of the seven orbits of
$\mathcal C_4$. Similarly, among all orbits of $\mathcal C_6$ there are
exactly 672 in which some (or equivalently, every) ray takes zero at some
singleton, and exactly seven orbits in which the rays take zero at two or
more singletons.

\subsection{Weight (incidence) distribution}\label{subsec:weight}

\begin{figure}%
\centering\begin{tikzpicture}
\input{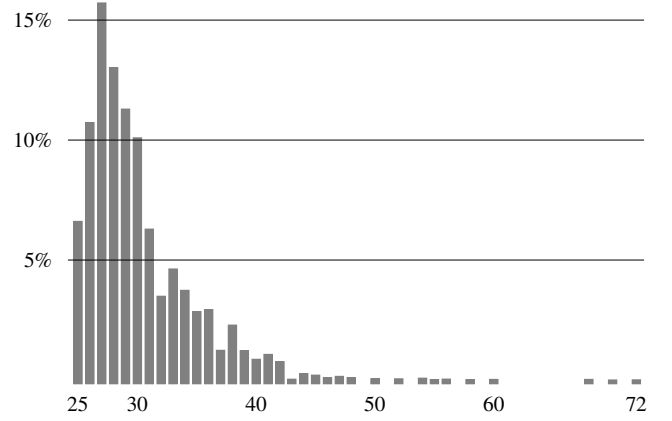}
\end{tikzpicture}%
\caption{Weight distribution for $n=5$}\label{fig:1}
\end{figure}

Weight distribution provides information about how degenerate the cone is. A
$d$-dimensional polyhedral cone is \emph{simple}, or \emph{non-degenerate},
if every extremal ray has the minimal possible weight $d-1$. Figure
\ref{fig:1} depicts the weight distribution of the extremal rays of
$\mathcal C_5$. The distribution has a long tail of rays with large weights.
The estimated weight distribution for $\mathcal C_6$ on Figure \ref{fig:2}
reveals a similar picture.

\begin{figure}%
\begin{tikzpicture}
\input{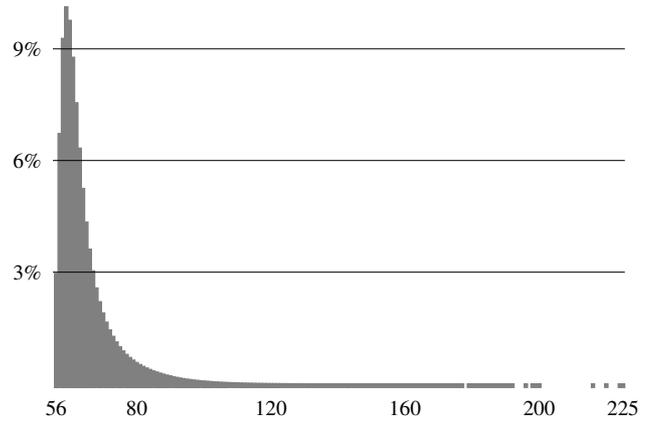}
\end{tikzpicture}%
\caption{Estimated weight distribution for $n=6$}\label{fig:2}%
\end{figure}

To trace the rays with large weights, consider the submodular function $f_J$
defined in (\ref{eq:fj}) for $J\subseteq X$, $|J|\ge 2$. This function is
the rank function of a connected matroid \cite{oxley}, and the rank function
of every connected matroid is extremal \cite{nguyen-78}. Consequently, $f_J$
is an extremal ray of $\mathcal C_n$. These rays (and their duals) have
large weights. Denoting $|J|$ by $k\ge 2$, $f_J$ satisfies all but
${k\choose 2}2^{n{-}k}$ of the elementary inequalities: the exceptional ones
are the triplets $(i,j|K)$ where $i,j\in J$ and $K$ is disjoint from $J$.
This means that $f_J$ has weight
$$
    {n\choose 2} 2^{n-2}-{k\choose 2}2^{n{-}k}.
$$
For $n=3,4,5$ the extremal rays with largest weight are the rays $f_J$ with
$|J|=2$ (and their symmetric versions); the corresponding weights are $4$,
$20$ and $72$, respectively, barely below the maximal possible weight values
of $5$, $23$, and $79$, which is one less than the number of facets. For
$n\ge6$, however, the largest weight is attained when $|J|=n$, that is, by
$f_X$. The weight of $f_X$ is ${n\choose 2}(2^{n-2}-1)$, which exceeds the
cone dimension by a factor of $n^2/8$.

The weight distributions depicted on Figures \ref{fig:1} and \ref{fig:2}
clearly show that the degeneracy of the submodular cones is caused not by a
few rays with large weights, but by the fact that almost all possible
weights occur across the complete weight range.

\subsection{Insertion order}\label{subsec:insertion-order}

As discussed in Section \ref{sec:eer}, conventional wisdom dictates to use
the Double Description method when the cone has a high degree of degeneracy,
see also \cite{avis2016} and \cite[Section 3]{polymake17}. The DD method,
however, is highly sensitive to the insertion order of the facets. Without
careful ordering, the number of extremal rays in an intermediate cone can be
exponentially larger than the final number of rays
\cite{avis2016,fukuda20,fukuda-prodon}, rendering the DD method unusable.
Insertion strategies are categorized by the way they are applied:
\emph{static orderings} are determined and applied in advance and the
ordering is fixed during the computation, while \emph{dynamic orderings}
allow for determining the order of the remaining inequalities dynamically,
depending on the state of the computation.

The static ordering strategy called \emph{lex-min} sorts the rows of the
matrix lexicographically, and applies them in increasing order. An example
for a dynamic strategy is \emph{max-cut} (resp.~\emph{min-cut}), which
chooses the unprocessed matrix row which cuts off the maximal
(resp.~minimal) number of extremal rays from the actual intermediate cone.
Experimental assessment of different strategies indicated, see
\cite{fukuda-prodon,howgood}, that any one can work reasonably well for some
enumeration problem, but could fail badly for others. The overall
recommendation is to use lex-min or some of its refinements, as this
strategy constantly outperforms all others, frequently by orders of
magnitude. See \cite[Chapter 8.1]{fukuda20} or \cite[Chapter 3.2]{terzer09}.

In the specific case of submodular cones, however, better insertion
strategies may exist, especially as the lex-min ordering is sensitive to the
order of the coordinates, while the DD method itself is not. Moreover,
previous experience showed that existing ordering strategies lead to an
``overshoot'' of the intermediate cones in terms of the number of extremal
rays, before this number reduced to the final result.
\begin{figure}[!htb]
\begin{tikzpicture}
\input{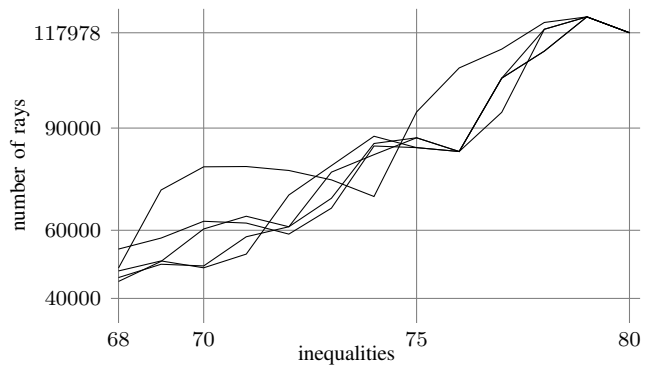}
\end{tikzpicture}%
\caption{Size of intermediate cones for lex-min inserting order using randomly shuffled coordinates}\label{fig:4b}
\end{figure}
This is illustrated by the size of the last few intermediate cones, depicted
on Figure \ref{fig:4b}, where the DD method was applied to $\mathcal C_5$
using the lex-min inserting strategy with several random permutations of the
coordinates. Our aim was to find an insertion strategy that avoids this
phenomenon.

The method to achieve this goal was a backward greedy algorithm, which
calculated the insertion order in \emph{reverse}. We first determined which
facet (which one out of the $m$ rows of the matrix $M$) should be added
\emph{last} so that the penultimate cone would have the smallest possible
number of extremal rays. Fortunately, this task requires solving at most $m$
ray enumeration problems, as the penultimate cone is independent of the
insertion order of its $m-1$ facets. Next, further elements of the insertion
order were determined similarly, always making sure that the initial
sequence of inequalities had full rank.

We determined the insertion order experimentally for $\mathcal C_5$, and
aimed to generalize the resulting order to $n=6$. To illustrate the choice
for the last inequality, note from Table \ref{table:1} that $\mathcal C_5$
has dimension $d=26$ and is bounded by $m=80$ facets. Also, the facets of
$\mathcal C_5$, or, equivalently, the rows of the defining matrix $M_5$,
correspond to the elementary inequalities $(i,j|K)$. Denote the elements of
the 5-element base set $X$ by digits from 0 to 4. Taking permutational
equivalence into account there are only four different elementary
inequalities, namely
$$
   (0,1|\emptyset),~~ (0,1|2),~~ (0,1|23), \mbox{ and } (0,1|234).
$$
As discussed in Section \ref{subsec:symmetry}, reflection (or duality) maps
$(0,1|\emptyset)$ to $(0,1|234)$, and maps $(0,1|2)$ to $(0,1|34)$, and the
latter one is permutationally equivalent to $(0,1|23)$. Consequently only
two cases has to be considered for the last position: it is either
$(0,1|\emptyset)$ or $(0,1|2)$. In the first case the remaining inequalities
determine the penultimate cone with 112\;712 extremal rays, in the second
case that cone has 122\;642 extremal rays. Therefore $(0,1|\emptyset)$ (or
one of its equivalents) should be inserted last.

We abstracted the patterns discovered in the experimental results into the
\emph{\tailo} (for ``tail-optimal'') \emph{insertion order} defined below.
Note that the experimental results do not determine the order of the first
$d$ inequalities.

For the definition of the \tailinsorder{}, fix an ordering of the
$n$-element base set $X$. A subset $K$ of $X$ is identified with the string
of length $|K|$ enlisting elements of $K$ in increasing order, and the
elementary inequality $(i,j|K)$ is always written so that $i$ precedes $j$,
similarly to the examples above.

\begin{figure}%
\begin{tikzpicture}
\input{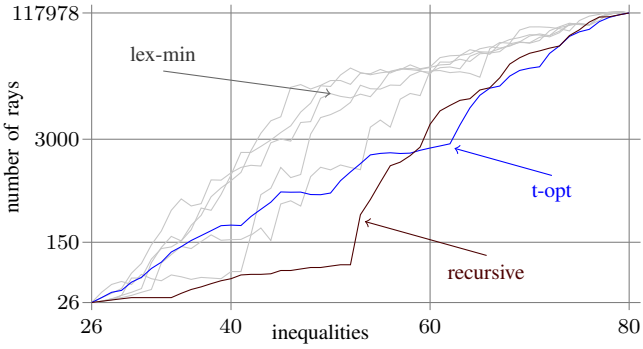}
\end{tikzpicture}
\caption{Size of intermediate cones for \tailo{} (blue), 
recursive (red), and lex-min (gray) insertion order for $n=5$}\label{fig:4}
\end{figure}

\begin{definition}[\Tailinsorder]\label{def:std}
Two elementary inequalities are in the relation $(i_1,j_1|K_1)\rstd(i_2,j_2|K_2)$
iff
\begin{itemz}
\item when the subsets $K_1$ and $K_2$ have different number of elements
then $|K_1|$ precedes $|K_2|$ in the list $0$,
$n-2$, $1$, $n-3$, $2$, $n-4$, \dots
\item when $K_1$ and $K_2$ have the same number of elements
then the string $i_1j_1K_1$ is lexicographically smaller than the string $i_2j_2K_2$.
\end{itemz}
The \emph{\tailinsorder} is the reverse of the order $\rstd$.
\end{definition}

For the base set $X=\{0,1,2,3\}$ the order $\rstd$ is illustrated by
$$
    (2,3|\emptyset)\rstd
    (1,2|03) \rstd
    (0,1|3) \rstd (1,3|2);
$$
moreover, $(0,1|\emptyset)$ is the smallest, and $(2,3|1)$ is the largest one
in this ordering.

The blue curve on Figure \ref{fig:4} depicts, on logarithmic scale, the size
of the intermediate cones when the \tailinsorder{} is used to enumerate the
extremal rays of $\mathcal C_5$. The curve increases steadily (as opposed to
the cases where the lex-min order was used), and seems to be approximately
linear, indicating a steady exponential growth. Using this insertion order
for $\mathcal C_5$ the total execution time of our implementation of the DD
method running on a single core of an Intel i5-4590 CPU was under 1 minute,
beating all other tested insertion strategies. For comparison, the same task
on the same single core using POLCO v.4.7.1 \cite{terzer09} took 1.9
minutes, and using \emph{lrs} (lrslib v.7.3, \cite{avis-lrs}) took 590
minutes.

When the \tailinsorder{} was applied to $\mathcal C_6$, the number of
extremal rays in the intermediate cones grew steadily, as expected. After
completing 44 iterations this number reached 18\;506\;227. At this point the
computation was stopped as the last iteration took over 200 hours to finish,
and the next iteration was estimated to require five to ten times as much
running time, see Section \ref{subsec:neig2}. The size of intermediate cones
up to this point is plotted as the blue curve on Figure \ref{fig:5}.
\begin{figure}%
\centering
\begin{tikzpicture}
\input{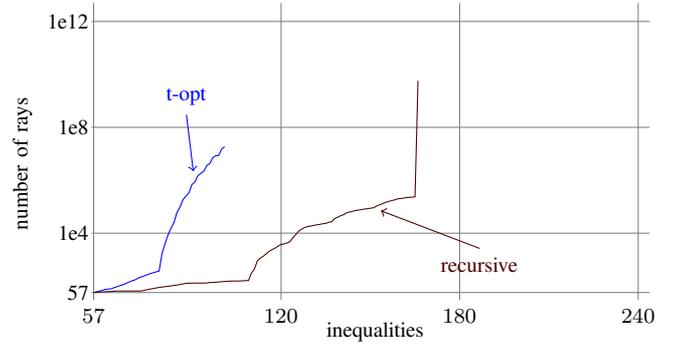}
\end{tikzpicture}
\caption{Size of intermediate cones for \tailo{} (blue) and
recursive (red) insertion order for $n=6$}\label{fig:5}
\end{figure}
Assuming it is similar to the $n=5$ case, this initial segment gives an
estimate for the total number of extremal rays of $\mathcal C_6$ somewhere
between $10^{20}$ and $10^{30}$.

\smallskip

Due to the excessive time the next iteration would have taken, the DD
algorithm for $\mathcal C_6$ and using the \tailinsorder{} had to be stopped
quite early. Therefore we searched for a different inserting regime which
would allow processing more inequalities in reasonable time. It was observed
that when $K_1\subset K_2\subset K_3\subset\cdots$ inserting the
inequalities
$$
   (i,j|K_1), ~~ (i,j|K_2), ~~ (i,j|K_3),~~~\dots
$$
in this order resulted in a moderate increase in the number of the extremal
rays. This observation motivated our definition of the \emph{recursive
order}. For the definition we assume again that the elements of $X$ are
ordered. The recursive procedure $\REC(K)$ enumerates the subsets of
$K\subseteq X$, where $K$ is specified as an increasing list of its
elements. The enumeration is defined as follows:
\begin{itemz}
\item List $K$ itself first;
\item For all $i$ in $K$ in decreasing order, call $\REC(K\setm i)$ 
and keep the first occurrence of each emitted subset.
\end{itemz}
That is, $\REC(0,1,2,3)$ lists $\{0,1,2,3\}$ first, then calls
$\REC(0,1,2)$, $\REC(0,1,3)$, $\REC(0,2,3)$ and $\REC(1,2,3)$. Also,
$\REC(0,1)$ produces $\{0,1\}, \{0\}, \emptyset, \{1\}$ in this order, and
$\REC(0,1,2)$ produces
$$
\{0,1,2\}, \{0,1\}, \{0\}, \emptyset, \{1\}, \{0,2\}, \{2\}, \{1,2\}.
$$

\begin{definition}[Recursive insertion order]\label{def:rec}
Enumerate the elementary inequalities as follows. First, take the
two-element subsets $ij$ of $X$ with $i<j$ in lexicographic order. For each
such pair $ij$, enumerate all subsets $K$ of $X\setm ij$ by calling 
$\REC(X\setm ij)$ and append $(i,j|K)$ to the list.

The \emph{recursive insertion order} is the reverse of this enumeration.
\end{definition}

\noindent
As an example, for $X=\{0,1,2,3\}$ the recursive enumeration of the
inequalities starts with $(0,1|23)$, $(0,1|2)$, $(0,1|\emptyset)$,
$(0,1|3)$, $(0,2|13)$, and ends with $(2,3|0)$, $(2,3|\emptyset)$,
$(2,3|1)$. The insertion order is the reverse: it starts with $(2,3|1)$ and
ends with $(0,1|23)$.

Figures \ref{fig:4} and \ref{fig:5} also show the performance of the
recursive insertion order. Overall, it is slightly worse than the
\tailorder{} (using about 15\% more time to generate all extremal rays of
$\mathcal C_5$), but it allows inserting significantly more inequalities
before the size of the intermediate cone suddenly increases. In case of
$\mathcal C_6$, as shown by Figure \ref{fig:5}, the last intermediate cone
before the size jump has 165 facets out of the 240, and has only 235\;961
extremal rays. In the next iteration the number of extremal rays jumps to
5\;733\;451\;485.

\medskip

Undoubtedly both insertion strategies are variants of lex-min. Nevertheless,
our experiments showed that the recursive order has some intricate
connection with the structure of the submodular cone. Denoting the smallest
element of the $n$-element base set $X$ by $0$, recall from Definition
\ref{def:rec} that the elementary inequalities $(0,j|K)$ form an end segment
of the recursive order. The intermediate cone $\mathcal C^*_n$ defined by
inserting all other inequalities has the moderate number of rays
\begin{equation}\label{eq:cstarrays}
  |\mathcal C^*_n| = 2|\mathcal C_{n-1}|+ (n{-}1).
\end{equation}
For example, $|\mathcal C^*_5|=78$ as $|\mathcal C_4|=37$ from Table
\ref{table:1}, while the total number of rays of $\mathcal C_5$ is 117\;978.
Similarly, $|\mathcal C^*_6|= 235\;961 = 2|\mathcal C_5|+5$, which is
negligible compared to the estimated number of rays of $\mathcal C_6$.
Equation (\ref{eq:cstarrays}) follows from the fact that the bounding
inequalities
$$
 \{\, (i,j|K),~ (i,j|0K) : i,j\in X\setm 0 \mbox{ and  }K\subseteq X\setm 0ij 
 \,\}
$$
of $\mathcal C^*_n$ define two disjoint copies of $\mathcal C_{n-1}$ on two
disjoint subsets of the coordinates (that is, subsets of $X$), namely those
which do not contain the element $0$, and those which do. The additional
$(n{-}1)$ term in (\ref{eq:cstarrays}) comes from the fact that the first
copy is not $p$-standardized, see Section \ref{subsec:std}. The recursive
nature of this insertion strategy implies that similar reductions occur
earlier, explaining the recursive pattern of the initial segments of the red
curves in Figures \ref{fig:4} and \ref{fig:5}.

\medskip
Next, consider the step in the DD method after---or applied to---$\mathcal
C^*_n$. It exhibits an intriguing structural property of the submodular
cone, which explains the sudden jump in the number of extremal ray in the
next iteration.
\begin{claim}\label{claim:cstar}
Suppose the DD method cuts the cone $\mathcal C^*_n$ by the inequality $(0,i|K)$.
Among the extremal rays of $\mathcal C^*_n$ there will be exactly one more positive
than negative rays. Moreover, all positive / negative ray pairs are adjacent.
\end{claim}
\begin{proof}
As was done above, split the coordinates of $\mathcal C_n$ (that is, subsets
of $X$) into two classes depending on whether they contain the minimal
element $0$. Put coordinates containing $0$ first, followed by the others.
An extremal ray of $\mathcal C^*_n$ has zeros exclusively at the coordinates
in the first part, or zeros exclusively in the second part. Furthermore, the
non-all-zero part is an extremal ray of $\mathcal C_{n-1}$. Therefore,
extremal rays of $\mathcal C^*_n$ are the concatenations of a zero and a
non-zero vector in either $\langle r,\zerovec\rangle$ or $\langle
\zerovec,r\rangle$ order where $r$ is an extremal ray of $\mathcal C_{n-1}$.
It is easy to check that the algebraic condition in Claim
\ref{claim:adjacency} for adjacency holds for all ``opposite'' ray pairs
$\langle r_1,\zerovec\rangle$ and $\langle \zerovec,r_2\rangle$, thus all
these pairs are actually adjacent extremal rays of $\mathcal C^*_n$, see
also \cite{ziegler}. Choosing the inequality $(0,i|K)$ as the next one to be
inserted by the DD method, the extremal rays of $\mathcal C^*_n$ are split
into positive, zero, and negative ones depending on which side of the
hyperplane corresponding to $\delta(0,i|K)$ they are on. Observe that for
every extremal ray $r$ of $\mathcal C_{n-1}$ we have
$$
    0 \le \delta(0,i|K)\cdot \langle r,\zerovec\rangle = 
          - \delta(0,i|K)\cdot \langle \zerovec,r\rangle .
$$
Consequently, either both $\langle r,\zerovec\rangle$ and $\langle
\zerovec,r\rangle$ are in the zero part, or the first one is in the
positive, and the second one is in the negative part---meaning that the
positive and negative parts have equal size. Moreover, each positive /
negative ray pair is adjacent, and so they produce a new ray in the next
iteration. We also need to account for the additional $n{-}1$ extremal rays
in the first copy of $\mathcal C_{n-1}$ in $\mathcal C^*_n$, of which
exactly one goes into the positive part, and the other $(n{-}2)$ go into the
zero part, proving the first statement of the claim. The same reasoning as
above shows that this single additional ray in the positive part is adjacent
to all the negative ones, proving the second claim.
\end{proof}

In summary, independently which of the $(0,i|K)$ inequalities is added next
to $\mathcal C^*_n$, the number of positive rays will be one more than the
number of negative rays, and each positive / negative ray pair will define
an extremal ray in the next iteration. Specifically for $n=6$, the next
inequality with which the recursive order intersects $\mathcal C^*_6$ is
$(0,5|34)$. Among the 235\;961 extremal rays of $\mathcal C^*_6$ the number
of positive, zero, negative ones relative to this inequality are 75\;719,
84\;524, and 75\;718, respectively. Thus the number of extremal rays in the
next iteration is
$$
  75\;719+84\;524 + 75\;719\times 75\;718 = \mbox{5\;733\;451\;485},
$$
as was claimed earlier. A ray of this cone has 57 coordinates. Using only
two bytes to store a coordinate value, the list of extremal rays of this
iteration would occupy more than 650 gigabytes. 

\subsection{Generating the first extremal rays of $\mathcal C_6$}\label{subsec:first}

Extremal rays of $\mathcal C_5$, as discussed in Section
\ref{subsec:extending}, automatically provide extremal rays of $\mathcal
C_6$ via the extension $r(A)=r(Ay)=r'(A)$ where $r'$ is extremal in
$\mathcal C_5$.

Our first attempt to generate further extremal rays was to use Claim
\ref{claim:rank1}. Rows of the generating matrix $M_6$ were chosen randomly
until the chosen rows formed a submatrix of rank $d-1$. The one-dimensional
solution $r$ of this homogeneous system was then checked against $M_6r\ge
0$, that is, whether $r\in\mathcal C_6$. If yes, then $r$ provided an
extremal ray of $\mathcal C_6$. While every extremal ray had a chance to be
found, mostly rays with extremely large weights (see Figure \ref{fig:2})
were generated, due to the high number of submatrices generating the same
ray. Running even for a considerable time, this method generated only a few
thousand \emph{essentially different} rays, that is, rays from different
orbits, see Section \ref{subsec:symmetry}. The efficiency improved only
marginally when we tried to restrict the search to the neighborhood of a ray
$r$ that was found earlier. It was done by fixing a $d{-}1$-row submatrix of
$M_6$ which determines this $r$. When choosing rows of $M_6$ which were to
generate the next ray, the first $d-2$ rows were chosen from this submatrix.
By Claim \ref{claim:adjacency} extremal rays generated this way are those
which are adjacent to $r$.

\smallskip

Another possibility to generate extremal rays of $\mathcal C_6$ is indicated
by the following observation. If $r$ is an extremal ray of any intermediate
subcone of the DD method so that $r$ is also an element of $\mathcal C_6$,
then $r$ is an extremal ray of $\mathcal C_6$. It is because every extremal
ray of an intermediate cone is either an extremal ray of the next iteration,
or it is outside of the next cone, and thus of $\mathcal C_6$.

We ran the DD algorithm for several steps, and checked which extremal rays
of the last iteration were actually elements of $\mathcal C_6$. Since by
Claim \ref{claim:2} extremal rays of $\mathcal C_6$ have non-negative
coordinates, it was worth checking if $r\ge 0$ before delving into the more
time consuming computation of $r\in\mathcal C_6$. Furthermore, the check
$r\in\mathcal C_6$ can be incorporated into the inner loop of the lastly
executed DD step. Code \ref{code:4} details the replacement of the loop at
lines \ref{c1-l6}--\ref{c1-l11} of Code \ref{code:1}. It postpones the
expensive adjacency test for those pairs which would otherwise produce an
extremal ray of $\mathcal C_6$.
\begin{pseudocode}{Modified inner loop in the last round of DD\label{code:4}}
\LComment{previous cone is to be split by the row $a\in M_6$}
\ForEach{$r_1$ positive / $r_2$ negative ray}
  \State \textbf{if} $|\hat r_1\cap\hat r_2|<d-2$ \textbf{then} continue\label{c4-l3}
  \State compute the ray $r={}$conic$(r_1,r_2)\cap a$
  \State \textbf{if} not $r\ge 0$ \textbf{then} continue
  \State \textbf{if} not $M_6r\ge 0$ \textbf{then} continue
  \If{$r_1$ and $r_2$ are adjacent}\label{c4-l7}
    \State report $r$ as extremal in $\mathcal C_6$
  \EndIf
\EndFor
\end{pseudocode}
Line \ref{c4-l3} is the quick precheck of the combinatorial test detailed in
Code \ref{code:2}. If $r_1$ and $r_2$ pass this test, then the potential new
ray is computed and checked for being an element of $\mathcal C_6$. The
adjacency test is performed only if this is the case, resulting in a
significant speed-up. This modified DD iteration was used on top of the
intermediate cones produced by the \tailinsorder{} using different choices
for the next inequality.

Starting from $\mathcal C^*_6$ as the intermediate cone, by Claim
\ref{claim:cstar} every positive / negative ray pair of this cone is
adjacent. Therefore the checks in lines \ref{c4-l3} and \ref{c4-l7} of Code
\ref{code:4} should not be executed at all when using the modified iteration
on top of $\mathcal C^*_6$.

Overall, these processes provided about a half million essentially different
extremal rays of $\mathcal C_6$, that is, representatives of that many
different orbits.

\subsection{Visiting the neighborhood}\label{subsec:neig2}

\emph{Adjacency decomposition} is a natural approach when the underlying
problem has many symmetries, see \cite{bremner09,Plos17,rehn10}. As
described in Section \ref{subsec:neighbor}, adjacency decomposition starts
with some initial extremal rays of $\mathcal C$, generates their neighbors,
then the neighbors of these neighbors, etc., until no new ray can be
generated. Adjacency decomposition requires solving several enumeration
problems in one less dimension, typically with much smaller number of
constraints. It can enumerate all extremal rays even in cases when the
original DD method would exhaust all available resources \cite{sikiric07}.
The efficiency is partially due to the fact that adjacency decomposition can
take advantage of cone symmetries, while the DD method cannot. It is so as
it sufficies to find the neighbors of a single representative from each
orbit (that is, symmetry class), as neighbors of rays from the same orbit
are the symmetrical images of neighbors of this representative. Adjacency
decomposition can also be applied recursively to the most difficult
subproblems. Typically those difficult subproblems have a large number of
symmetries as well, improving the efficiency further.

A clear resource limit of the DD method is the total number of the extremal
rays to be enumerated. The same limit for the adjacency decomposition method
is the number of orbits. According to Figure \ref{fig:6}, in case of the
submodular cone $\mathcal C_6$, each orbit contains, with minimal
exceptions, $2n!=1440$ rays. Consequently adjacency decomposition could
reduce the complexity of enumerating extremal rays by three orders of
magnitude.

\begin{figure}%
\begin{tikzpicture}
\draw (-0.2,3.36441165607601) node[left] {\footnotesize 1e-3};
\draw (-0.15,3.36441165607601) -- (6.5,3.36441165607601);
\draw (-0.2,2.11785772970005) node[left] {\footnotesize 1e-5};
\draw (-0.15,2.11785772970005) -- (6.5,2.11785772970005);
\draw (-0.2,0.871303803324086) node[left] {\footnotesize 1e-7};
\draw (-0.15,0.871303803324086) -- (6.5,0.871303803324086);
\node[rotate=90,above=7pt] at (-0.6,6.5*0.45) {\hbox to 0pt{\hss\footnotesize relative frequency}};
\draw (0.21125,-0.4) node {\footnotesize 720};
\draw (0.86125,-0.4) node {\footnotesize 360};
\draw (1.51125,-0.4) node {\footnotesize 240};
\draw (2.16125,-0.4) node {\footnotesize 180};
\draw (2.81125,-0.4) node {\footnotesize 120};
\draw (3.46125,-0.4) node {\footnotesize 90};
\draw (4.11125,-0.4) node {\footnotesize 60};
\draw (4.76125,-0.4) node {\footnotesize 40};
\draw (5.41125,-0.4) node {\footnotesize 30};
\draw (6.06125,-0.4) node {\footnotesize $\le$20};
\draw (6.5*0.45,-0.8) node {\footnotesize orbit size};
\filldraw[gray] (0,0) rectangle (0.4225,3.6);
\filldraw[gray] (0.65,0) rectangle (1.0725,2.72404012605833);
\filldraw[gray] (1.3,0) rectangle (1.7225,1.95896871995725);
\filldraw[gray] (1.95,0) rectangle (2.3725,1.82199065458931);
\filldraw[gray] (2.6,0) rectangle (3.0225,1.75571127059862);
\filldraw[gray] (3.25,0) rectangle (3.6725,1.19283608602081);
\filldraw[gray] (3.9,0) rectangle (4.3225,1.05892886485002);
\filldraw[gray] (4.55,0) rectangle (4.9725,0.848733590321585);
\filldraw[gray] (5.2,0) rectangle (5.6225,0.946454757005704);
\filldraw[gray] (5.85,0) rectangle (6.2725,0.871303803324086);
\end{tikzpicture}%
\caption{Relative frequency of orbits of $\mathcal C_6$ with size below
 1440.}\label{fig:6}%
\end{figure}
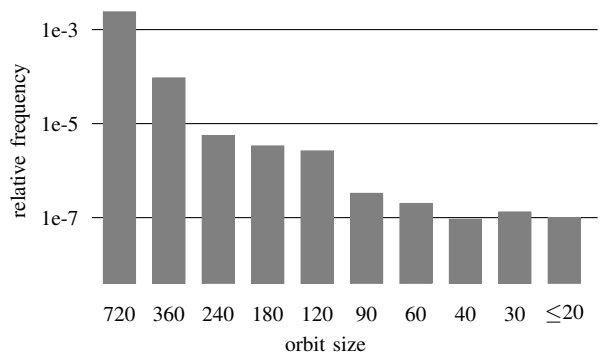

Enumerating neighbors should be started with low-weight rays as in these
``easy'' cases the DD method can be used directly as outlined in Section
\ref{subsec:neighbor}. This approach was successfully used to generate 260M
essentially different extremal rays of $\mathcal C_6$ (that is, rays on
different orbits), starting from the rays computed earlier. This required
computing neighbors of only about 450\;000 rays. These rays had typically
low weights: 69\% had weights between 56 and 59, 21\% had weights 60--69,
and the remaining 10\% had weights between 70 and 80. On average, each
probed ray produced over 600 neighbors. More than 95\% of these computations
took less than 1 second. Figure \ref{fig:3} depicts a sample of longer
computations showing the total number of generated neighboring rays versus
the speed of generation on the same CPU and using a single core.
\begin{figure}[!b]%
\begin{tikzpicture}
\input{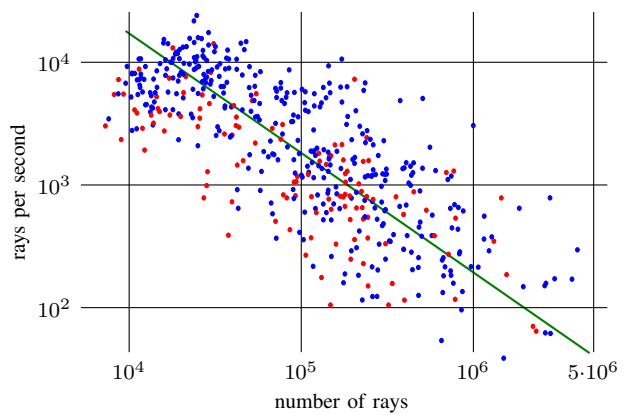}
\end{tikzpicture}%
\caption{Speed of generating adjacent rays. Blue: \tailorder, red:
recursive order.}\label{fig:3}%
\end{figure}
The computation used either the \tailorder{} (blue dots) or the recursive
order (red dots), see Section \ref{subsec:insertion-order}, restricted to
the rows of the corresponding submatrix. While the recursive order seems to
appear mainly in the slower region, in some cases it was significantly
faster than the \tailorder. The plotted data seem to follow, alas with large
deviation, a linear trend marked by the green line. Since both coordinates
are on logarithmic scale, it gives the exponential approximation
$$
    \mbox{speed} = C\cdot \mbox{size}^{-0.6}
$$
for some constant $C$. Thus if the output size increases tenfold, then
the speed goes down by a factor of $10^{-0.6}\approx 0.25$, and
the total generating time is expected to go up by a factor of $40$.

\smallskip

The weight distribution for rays and for orbits are almost identical
as only a small fraction of the orbits is not maximal, see Figure \ref{fig:6}.
Using the data of Figure \ref{fig:2}, 90\% of the weights are expected
to be below 75, and so to belong to the ``easy'' cases.
\begin{table}[!htb]%
\def\p(#1){~(#1)}%
\def\xx#1{}%
\caption{Number of neighboring orbits of the heaviest extremal rays
for $n=5$ (left), and $n=6$ (right)}\label{table:2}%
\hfill
\begin{tabular}[t]{|c|c|c|}
\hline
\rule[-5pt]{0pt}{15pt}$|J|$ & $w(f_J)$ & orbits \\
\hline
\rule{0pt}{10pt}%
2     &   72     \xx{45149 \p(38)} & \!100\% \p(672) \\
3     &   68     \xx{25778 \p(22)} & ~99\% \p(664) \\
4     &   68     \xx{27792 \p(24)} & ~95\% \p(636) \\
5     &   70     \xx{24674 \p(21)} & ~44\% \p(299) \\
\hline
\end{tabular}%
\qquad
\begin{tabular}[t]{|c|c|c|}
\hline
\rule[-5pt]{0pt}{15pt}$|J|$ & $w(f_J)$ & orbits \\
\hline\rule{0pt}{10pt}%
2 & 224 & 89\% \\
3 & 216 & 83\% \\
4 & 216 & 89\% \\
5 & 220 & 76\% \\
6 & 225 & 37\% \\
\hline
\end{tabular}\hfill\hbox{}
\end{table}
There are, however, ``difficult'' cases as well. As discussed in Section
\ref{sec:eer}, almost all weights occur between the smallest and the largest
possible values with example of rays with large weights being the extremal
rays $f_J$ defined in (\ref{eq:fj}). Table \ref{table:2} lists how many
neighboring orbits the rays $f_J$ have. For $n=5$ the exact numbers are
shown in parentheses, for $n=6$ the percentages are estimates. The data
implies that even listing the neighboring orbits of these extremal rays
requires space comparable to the total number of orbits. For $n=6$ finding
all neighbors of $f_{01}$, or at least checking whether all of them have
been found, would require efforts comparable to enumerating all orbits of
the extremal rays of $\mathcal C_6$.

\section{Estimating the total number of rays and orbits for $n=6$}\label{sec:estimate}

While we successfully generated a large number of extremal rays of $\mathcal
C_6$, generating a complete list with any of the above methods would require
an unrealistic amount of resources. We still aimed to provide an estimate
for their overall number to understand the expected complexity of the
complete problem. However, to the best of our knowledge, no efficient
randomized algorithm exists for computing a reasonable approximation of the
total number of extremal rays of a polyhedral cone, and only two general
approaches have been proposed. The first method, developed by Avis \&
Devroye \cite{avis-count}, is based on the backtrack tree size estimator of
Knuth, and was implemented around the reverse search (RS) vertex enumerating
algorithm \cite{avis-fukuda}. This estimator, while theoretically unbiased,
has enormous variance and in many cases vastly underestimates the number of
rays. The implementation in the software package \emph{lrs} \cite{avis-lrs}
gave the estimates ${\approx}12\;000$ for $\mathcal C_5$ and
${\approx}102\;000$ for $\mathcal C_6$ consistently. (The parameter
\emph{maxdepth} was set to 2 and the iteration count was set to 100.) The
fact that $\mathcal C_5$ has almost ten times as many rays renders these
estimates too inaccurate to be useful.

The second estimation method from \cite{salomone-count} is based on MCMC
(Markov Chain Monte Carlo) and uses conditional sampling to reduce the
variance. The algorithm, however, assumes that choosing $(d{-}1)$ rows of
the generating matrix $M$ randomly samples its maximal $(d{-}1)$-rank
subsets uniformly. This is not the case for the highly degenerate submodular
cones. In order to be able to use this estimation method, an effective
sampler of the maximal, $(d{-}1)$-rank subsets of $M$ would have to be
devised first. This task seems to be an equally difficult and challenging
problem. With such a sampler our first attempt to create extremal rays of
$\mathcal C_6$ as described in Section \ref{subsec:first} would have been
significantly more efficient.

To get an estimate on the total number of extremal rays of $\mathcal C_6$ we
resort to an heuristic argument without theoretical guarantee. From the
existing pool of 260M essentially different rays, 1000 were chosen randomly
with the following restrictions: the ray was not used in the pool creation,
and the ray weight is at most 70. According to Figure \ref{fig:2} these
restrictions exclude about 15\% of rays. All neighbors of the chosen 1000
rays were computed, producing 2\;824\;119 extremal rays. These rays
determined 2\;797\;684 distinct orbits (99\%), of which 154\;170 (about
5.5\%) have already had a representative in the pool. Assuming that the pool
is a completely random subset of all orbits, this yields the estimate
260M\,/\,0.055${}\approx 4.7{\cdot}10^{9}$ for the total number of orbits.
From here, based on the weight distribution depicted on Figure \ref{fig:2},
the number of extremal rays is estimated to be $6.5{\cdot}10^{12}$.
Repeating this experiment gave similar results.

This estimate is clearly biased as the pool is far from random: it was
created by adding neighbors, in several stages, to a relatively small number
of rays. While the above figures are quite reasonable, we expect the actual
count to be significantly larger, closer to the lower estimate $10^{20}$
obtained in Section \ref{subsec:insertion-order}.

\section*{Acknowledgment}

The research reported in this paper has been partially supported by the ERC
Advanced Grant ERMiD.

\bibliography{refer}

\end{document}